\documentclass [12pt,a4paper]{article}
 \usepackage{amssymb}
 \usepackage{amsmath,amsthm}
 \usepackage{mathrsfs}
\usepackage[numbers,sort&compress]{natbib}
 \usepackage{graphicx}
\usepackage{tikz}
 \usepackage[colorlinks=true,citecolor=blue]{hyperref}
 \usepackage{bbm}

\normalsize

\renewcommand{\dfrac}[2]{\lower0.15ex\hbox{\large$\textstyle\frac{#1}{#2}$}}

\newtheorem{theorem}{Theorem}[section]
\newtheorem{lemma}[theorem]{Lemma}

\newtheorem{corollary}[theorem]{Corollary}
\newtheorem{definition}[theorem]{Definition}

\newtheorem{remark}[theorem]{Remark}

\numberwithin{equation}{section}

\begin{document}
\title 
{\bf On the number of linear uniform hypergraphs with linear girth constraint}
\author
{{Fang Tian$^1$\quad Yiting Yang$^2$\quad Xiying Yuan$^3$}\\ 
{\small $^1$Department of Applied Mathematics}\\%
{\small Shanghai University of Finance and Economics, Shanghai, 200433, China} \\
{\small\tt tianf@mail.shufe.edu.cn}\\[1ex]
{\small $^2$School of Mathematical Sciences}\\
{\small Tongji University, Shanghai, 200092, China}\\
{\small\tt ytyang@tongji.edu.cn}\\
{\small $^3$Department of Mathematics}\\
{\small Shanghai University, Shanghai, 200444, China}\\
{\small\tt xiyingyuan@shu.edu.cn}
}
\date{}
 \maketitle


\begin{abstract}
For an integer $r\geqslant 3$, a hypergraph on vertex set $[n]$ is $r$-uniform
 if each edge is a set of $r$ vertices, and is said to be linear if every 
 two distinct edges share at most one vertex. Given a family 
 $\mathcal{H}$ of linear $r$-uniform hypergraphs,
 let $\textmd{Forb}_r^L(n,\mathcal{H})$ be the set of linear $r$-uniform hypergraphs
  on vertex set $[n]$, which does not contain any member from $\mathcal{H}$ as a subgraph.
  An $r$-uniform linear cycle of length $\ell$, denoted by  $C_\ell^r$, is a 
  linear $r$-uniform hypergraph on $(r-1)\ell$ vertices whose edges can be ordered as 
 $\boldsymbol{e}_1,\ldots,\boldsymbol{e}_\ell$ such that 
 $|\boldsymbol{e}_i\cap \boldsymbol{e}_j|=1$ if $j=i\pm 1$ (indices taken modulo $\ell$)
  and $|\boldsymbol{e}_i\cap \boldsymbol{e}_j|=0$ otherwise. The 
  \textit{linear girth} of a linear $r$-uniform hypergraph
is the smallest integer $\ell$ such that it contains a $C_\ell^r$.
Let $\textmd{Forb}_L(n,r,\ell)=\textmd{Forb}_r^L(n,\mathcal{H})$  when
$\mathcal{H}=\{C_i^r:\, 3\leqslant i\leqslant \ell\}$, that is, 
$\textmd{Forb}_L(n,r,\ell)$ is the set of all linear $r$-uniform hypergraphs on $[n]$
with linear girth greater than $\ell$. For integers $r\geqslant 3$
and $\ell\geqslant 4$, Balogh and Li [On the number of linear hypergraphs
 of large girth, J. Graph Theory, {\bf 93}(1) (2020), 113-141] 
showed that
$|\textmd{Forb}_L(n,r,\ell)|= 2^{O(n^{1+1/\lfloor \ell/2\rfloor})}$
 based on the graph container method. 
It is natural to obtain 
$|\textmd{Forb}_L(n,r,\ell)|\geqslant 2^{c\cdot n^{1+1/\ell}}$
for some constant $c$ by probabilistic deletion method. 
Combined with the known results that
$|\textmd{Forb}_L(n,r,3)|= 2^{o (n^{2})}$
and $|\textmd{Forb}_L(n,3,4)|= 2^{\Theta (n^{3/2})}$, 
by analyzing the  random greedy high {linear} girth linear $r$-uniform hypergraph process,
we show 
$|\textmd{Forb}_L(n,r,\ell)|\geqslant 2^{n^{1+1/(\ell-1)-O(\log\log n/\log n)}}$
for every pair of fixed integers $r,\ell\geqslant 4$, or $r= 3$ and $\ell\geqslant 5$.
\end{abstract}




\section{Introduction}

For an integer $r\geqslant 3$,
 a hypergraph $H=(V_H,E_H)$ on vertex set $[n]$ is an
 \textit{$r$-uniform hypergraph} (\textit{$r$-graph} for short) if each edge 
  is an $r$-element subset (\textit{$r$-set} for short) of $[n]$,
 and is said to be \textit{linear} if any two distinct edges share at most
 one vertex. Write ${v}_H$ and ${e}_H$ to denote the number of vertices and 
 edges of $H$. An $r$-graph $H$ is called a partial Steiner $(n,r,s)$-system
 for an integer $s$ with $2\leqslant s\leqslant r-1$,
 if every $s$-set lies in at most one 
 edge of it. In particular,  partial Steiner $(n,r,2)$-systems are
 linear $r$-graphs, partial Steiner $(n,3,2)$-systems are partial 
 Steiner triple systems. 
Given a family $\mathcal{H}$ of linear $r$-graphs, define the 
\textit{linear Turán number} of $\mathcal{H}$, denoted by 
 $\textmd{ex}_r^L(n,\mathcal{H})$, is the maximum 
number of edges among  linear $r$-graphs on vertex set $[n]$ that 
 contain no $r$-graphs of $\mathcal{H}$ as subgraphs ($\mathcal{H}$-\textit{free} for short). 
 Recently, the research on the Turán problem in linear hypergraphs has attracted 
 considerable attention. Let
  $\textmd{Forb}_r^L(n,\mathcal{H})$ be the set 
 of all  linear $\mathcal{H}$-free $r$-graphs on
vertex set $[n]$. 
When $\mathcal{H}$ consists of a single linear
 $r$-graph $H$, we simply write them
as $\textmd{ex}_r^L(n,{H})$  and 
$\textmd{Forb}_r^L(n,{H})$ instead. The Tur\'{a}n problem and the enumeration problem 
for a family $\mathcal{H}$ of $r$-graphs are closely related. Every subgraph of a $\mathcal{H}$-free graph
is also $\mathcal{H}$-free, then a trivial bound on the size of  $\textmd{Forb}_r^L(n,\mathcal{H})$
is given in~\cite{balgoh17} as follows
\begin{align}
2^{\textmd{ex}_r^L(n,\mathcal{H})}\leqslant |\textmd{Forb}_r^L(n,\mathcal{H})|\leqslant 
\sum_{i\leqslant \textmd{ex}_r^L(n,\mathcal{H})}\binom{\binom{n}{r}}{i}\leqslant 2n^{r\cdot \textmd{ex}_r^L(n,\mathcal{H})}.
\end{align}

 There are several different notions of hypergraph cycles.
  For an integer $\ell\geqslant 3$, a Berge cycle $B_\ell^r$
  of length $\ell$ is an alternating sequence of distinct vertices and distinct 
  edges of the form ${v}_1,\boldsymbol{e}_1,{v}_2,\boldsymbol{e}_2,\ldots,{v}_\ell,\boldsymbol{e}_\ell$,
  where ${v}_i,{v}_{i+1}\in \boldsymbol{e}_i$ for each $i\in[\ell-1]$ and ${v}_\ell,{v}_1\in \boldsymbol{e}_\ell$.
 This definition
  of a hypergraph cycle is the classical definition due to Berge.
    Note that forbidding a Berge cycle $B_\ell^r$ actually forbids a family of hypergraphs.
    Another notion of hypergraph cycles that has been actively investigated 
  is that of a linear cycle, which is unique up to isomorphism. An $r$-uniform linear cycle  $C_\ell^r$ of length $\ell$ 
 is an alternating sequence ${v}_1,\boldsymbol{e}_1,{v}_2,\boldsymbol{e}_2,\ldots,{v}_\ell,\boldsymbol{e}_\ell$ of distinct vertices and distinct 
  edges such that $\boldsymbol{e}_i\cap \boldsymbol{e}_{i+1}=\{{v}_{i+1}\}$ for each $i\in[\ell-1]$,
   $\boldsymbol{e}_1\cap \boldsymbol{e}_\ell=\{{v}_1\}$ and $\boldsymbol{e}_i\cap \boldsymbol{e}_j=\emptyset$ if $1<|j-i|<\ell-1$. 
The \textit{linear girth} of a linear $r$-graph is the smallest integer $\ell$ 
such that it contains a $C_\ell^r$. If no such $\ell$ exists, 
the {linear} girth of the linear $r$-graph is defined to be infinite.

Determining $\textmd{ex}_3^L(n,C_3^3)$ is equivalent to the 
famous $(6,3)$-problem, which aims at determining the maximum number of edges
of triple systems not carrying three edges on six vertices. The $(6,3)$-problem is studied 
by Brown, Erd\H{o}s and S\'{o}s~\cite{brown73}, and it is known that for some constant $c>0$,
\begin{align*}
n^{2- {c}{\sqrt{\log n}}}<\textmd{ex}_3^L(n,C_3^3)=o(n^2),
\end{align*}
where the lower bound is given by Behrend~\cite{behrend46} and the upper bound is given by 
Rusza and Szemer\'{e}di~\cite{ruz78}. 
Erd\H{o}s, Frankl and R\"{o}dl~\cite{erd86}
showed that for every $r\geqslant 3$, $\textmd{ex}_r^L(n,C_3^r)=o(n^2)$ and 
$\textmd{ex}_r^L(n,C_3^r)>n^{2-o(1)}$. 
Using the so-called $2$-fold Sidon sets,
Lazebnik and  Verstra\"{e}te~\cite{laz03} constructed linear $3$-graphs with girth $5$
and showed that  
\begin{align*}
\textmd{ex}_3^L\bigl(n,\{B_3^3, B_4^3\}\bigr)= \frac{1}{6}n^{  \frac{3}{2}}+O(n).
\end{align*}
Kostochka, Mubayi and Verstra\"{e}te asked if for integers $r\geqslant 3$ and $\ell\geqslant 4$,
\begin{align*}
\textmd{ex}_r^L\bigl(n,C_\ell^r\bigr)=\Theta\bigl(n^{1+ \frac{1}{\lfloor \ell/2\rfloor}}\bigr).
\end{align*} 
Later, Collier-Cartaino, Graber and Jiang~\cite{coll14} 
proved that $\textmd{ex}_r^L(n,C_\ell^r)=O(n^{1+ \frac{1}{\lfloor \ell/2\rfloor}})$.
 In detail,
they obtained there exist positive constants $c(r, k)$ and $d(r, k)$
 such that 
 \begin{align*}
 \textmd{ex}_r^L(n,C_{2k}^r)\leqslant c(r,k)n^{1+ \frac{1}{k}}\ \text{and }
 \textmd{ex}_r^L(n,C_{2k+1}^r)\leqslant d(r,k)n^{1+ \frac{1}{k}},
 \end{align*}
where the integer $k\geqslant 2$, the constants $c(r, k)$ and $d(r, k)$ are 
exponential in $k$ for fixed $r$.
As an immediate corollary of the main results in~\cite{jiang30}, Jiang, Ma and Yepremyan 
further improved the coefficient $c(r, k)$ to be linear in $k$. 
Ergemlidze et al.~\cite{ergem2019} strengthened some results
in $3$-graphs to be
$\textmd{ex}_3^L(n, B_4^3)= \frac{1}{6}n^{  3/2}+O(n)$, $\textmd{ex}_3^L(n, B_5^3)= \frac{1}{3\sqrt{3}}n^{  3/2}+O(n)$,
and for $k=2,3,4$ and $6$,
\begin{align*}
\textmd{ex}_3^L\bigl(n,\{C_3^3,C_5^3,\ldots,C_{2k+1}^3\}\bigr)= \Theta\bigl(n^{1+ \frac{1}{k}}\bigr).
\end{align*}
In general,  for $k\geqslant 2$, they also showed that 
\begin{align*}
\textmd{ex}_3^L\bigl(n,\{C_3^3,C_5^3,\ldots,C_{2k+1}^3\}\bigr)= \Omega\bigl(n^{1+ \frac{2}{3k-4+\epsilon}}\bigr),
\end{align*}
where $\epsilon=0$ if $k$ is odd and $\epsilon=1$ if $k$ is even.
As a special case of the linear hypergraph extension of K\H{o}v\'{a}ri–S\'{o}s–Tur\'{a}n's theorem in~\cite{gao2021},
Gao and Chang showed that $\textmd{ex}_3^L(n, C_4^3)= \frac{1}{6}n^{  3/2}+O(n)$, 
and Gao et al. in~\cite{gao2023} further proved that $\textmd{ex}_3^L(n,C_5^3)= \frac{1}{3\sqrt{3}}n^{  3/2}+O(n)$.
For the lower bound of $\textmd{ex}_r^L(n,C_\ell^r)$, the best known lower bound  for $r=3$ is 
$\textmd{ex}_3^L\bigl(n,C_\ell^3\bigr)=\Omega(n^{1+ \frac{1}{\ell-1}})$
that was mentioned in~\cite{coll14,ergem2019} due to Verstra\"{e}te, by taking a random subgraph of a Steiner
triple system. 
The lower bound on the linear Tur\'{a}n number of linear cycles is still
far from what is conjectured, especially in a higher uniformity $r\geqslant 4$,
 which is an area worth some exploration~\cite{ergem2019}.
Sometimes to explore the lower bound of the corresponding
Tur\'{a}n problem is more difficult than the upper one.

Following the same logic with the Tur\'{a}n problems of cycles, 
for integers $r\geqslant 3$ and $\ell\geqslant 4$, Balogh and Li~\cite{balgoh17} conjectured that 
$ |\textmd{Forb}_r^L(n,C_\ell^r)|=2^{\Theta(n^{1+1/\lfloor \ell/2\rfloor})}$.
For $\ell=3$, the work of Erd\H{o}s, Frankl and R\"{o}dl~\cite{erd86} 
could be extended to show that $|\textmd{Forb}_r^L(n,C_3^r)|=2^{o(n^{2})}$. 
Balogh and Li~\cite{balgoh17} confirmed the conjecture in the case of $\ell=4$. In
fact, they proved that for every $r\geqslant 3$ there exists $C=C(r)>0$ such
that $|\textmd{Forb}_r^L(n,C_4^r)|\leqslant 2^{C\cdot n^{3/2}}$. The upper bound for $C_4^3$ is sharp
in order of magnitude given by $\textmd{ex}_r^L(n,C_4^3)=\Theta(n^{3/2})$ and (1.1).
In general, for $r\geqslant 3$ and $\ell\geqslant 4$, they provided a result on the girth version, that is,
there exists a constant $c=c(r,\ell)$ such that 
the number of linear $r$-graphs with {linear} girth greater than $\ell$ is at most 
$2^{c\cdot n^{1+1/\lfloor \ell/2\rfloor}}$. 
Define $\textmd{ex}_L(n,r,\ell)$ and $\textmd{Forb}_L(n,r,\ell)$
as  $\textmd{ex}_r^L(n,\mathcal{H})$ and $\textmd{Forb}_r^L(n,\mathcal{H})$ of
$\mathcal{H}=\{C_i^r:\, 3\leqslant i\leqslant \ell\}$, respectively.
Hence, it implies that 
$|\textmd{Forb}_L(n,r,\ell)|=2^{O(n^{1+1/\lfloor \ell/2\rfloor})}$. 
They stated that
the upper bounds are possibly sharp, while they are not able to confirm it now.
The idea of the proof in~\cite{balgoh17} is to reduce the 
 hypergraph enumeration problems to some graph enumeration problems, and 
then the graph container method is implemented in them.
Balogh, Narayanan and Skokan~\cite{bal18} made use of 
the hypergraph container method and provided a balanced supersaturation theorem
for linear cycles to  resolve a conjecture on the 
enumeration of $r$-uniform hypergraphs with a forbidden $C_\ell^r$.

In order to solve a conjecture of Erd\H{o}s about the existence
of sparse Steiner triple systems,
Glock, K\"{u}hn, Lo and Osthus~\cite{glock20}, and 
Bohman and Warnke~\cite{boh19}, 
studied a constrained triple process along 
similar lines, which can be
seen as a generalization of the \textit{random greedy triangle removal process}~\cite{bohman15},
to iteratively build a sparse partial Steiner triple system,
that is,
adding the sparseness constraint does not affect 
the evolution of the process significantly.
To be precise, a Steiner triple system is called $\ell$-sparse,
or girth greater than $\ell$, if it does not
contain any $(i+2,i)$-configurations for integers $2\leqslant i\leqslant \ell$,
where a $(i+2,i)$-configuration 
is a set of $i$ triples  which span at most $i+2$ vertices.
Erd\H{o}s made his conjecture
in Steiner triple systems.
 Kwan, Sah, Sawhney and Simkin~\cite{kwan2022} proved Erd\H{o}s’s conjecture with proof
   built upon the approximate results of~\cite{boh19,glock20} 
  and the iterative absorption method, combined with many new ideas.
Erd\H{o}s's conjecture for Steiner systems in higher uniformities 
was conjectured in several papers, such as
~\cite{fure13,glock20,kee20}.
Delcourt and Postle~\cite{delcourt2022} and independently Glock,
Joos, Kim, K\"{u}hn and Lichev~\cite{glock2024}
settled approximate versions of the conjecture
made in~\cite{fure13,glock20,kee20}. 
Both approaches offer different benefits, and each has led to substantial subsequent work. 
In a very recent breakthrough,
   Delcourt and Postle~\cite{delcourt2024} proved Erd\H{o}s’s conjecture for general Steiner systems.

The
 \textit{random greedy $r$-clique removal process} starts with a complete graph
 on vertex set $[n]$, denoted as $\mathbb{G}(0)$,
 and $\mathbb{G}(i+1)$ is the remaining graph from $\mathbb{G}(i)$ by
 selecting one $r$-clique uniformly at random out of all $r$-cliques in $\mathbb{G}(i)$
 and deleting all its edges from the edge set ${E}(i)$ of $\mathbb{G}(i)$.
 The process terminates once the remaining graph contains no $r$-cliques,
 and along the way produces a linear $r$-graph, where it starts with the 
 empty $r$-graph on vertex set $[n]$, iteratively adding the vertex set of each 
 chosen $r$-clique as a new edge in the hypergraph.  
 Let $M=\min\{i\,:\mathbb{G}(i)\ {\mbox{is\ } r{\mbox {-clique free}}}\}$,
 and  ${E}(M)$ be the set of edges left unsaturated 
 by the produced $r$-cliques, 
 which are related via
 \begin{align*}
 |E(M)|= \binom{n}{2}- \binom{r}{2}M.
 \end{align*}
 In particular, this problem attracted much 
 attention when $r=3$ and it is also called the random greedy triangle removal process~\cite{bohman15}.  
Fewer results are known on $r\geqslant 4$.
Bennett and Bohman~\cite{bennett15} considered the random greedy
hypergraph matching process, which can be seen as a generalization
of the $r$-clique removal process. A special case of the conjecture proposed
in~\cite{bennett15} is \textit{w.h.p.} $|E(M)|=n^{2- \frac{2}{r+1}+o(1)}$.
 Their upper bound still
has not the correct order of magnitude. Without the appropriate substructures,
it seems impossible to rely on self-correcting behavior to determine
the order of $|E(M)|$.
Tian et al.~\cite{tian23} directly discussed the structure 
of the random greedy $r$-clique removal process to the same natural barrier
with the one in~\cite{bennett15}, that is, $|E(M)|\leqslant n^{2- \frac{1}{r(r-1)-2}+o(1)}$.
 This upper bound on $|E(M)|$ equals $n^{ \frac{7}{4}+o(1)}$
when $k=3$, coinciding with an upper bound obtained  in~\cite{grable97,bohman101}
for random triangle packing.
Joos and K\"{u}hn~\cite{joos2025} finally solved
this conjecture in a significantly strong form. 
Therefore, the maximum number of edges in a linear $r$-graph on $n$ vertices lies
between $\binom{n}{2}/ \binom{r}{2}-O(n^{2- \frac{2}{r+1}+o(1)})$ and 
$ \binom{n}{2}/ \binom{r}{2}$. 

By classical probabilistic deletion method, 
for every pair of integers $r,\ell\geqslant 3$, it is natural to prove that there exists a linear $r$-graph on vertex set $[n]$
that has at least $c\cdot n^{1+ \frac{1}{\ell}}$ edges for some constant $c$, such as $c= \frac{1}{2r(r-1)}$, and the linear girth
greater than $\ell$, which implies that
$\textmd{ex}_L(n,r,\ell)\geqslant c\cdot n^{1+ \frac{1}{\ell}}$ and 
$|\textmd{Forb}_L(n,r,\ell)|\geqslant 2^{c\cdot n^{1+ \frac{1}{\ell}}}$. 
In fact, it also shows that $\textmd{ex}_r^L(n,C_\ell^r)\geqslant c\cdot n^{1+ \frac{1}{\ell}}$
and $ |\textmd{Forb}_r^L(n,C_\ell^r)|
\geqslant 2^{c\cdot n^{1+ \frac{1}{\ell}}}$.
We choose a random subgraph of a linear
$r$-graph, where each edge is generated independently with probability
$p=n^{-\alpha}$ with $\alpha= 1- \frac{1}{\ell}$. The expected number of edges and
the expected number of linear cycles of length
no more than $\ell$ are approximately $ \frac{1}{r(r-1)}n^{2-\alpha}$
and $\Theta(n^{\ell(1-\alpha)})$, respectively. 
We delete one edge from each of these linear cycles to obtain one new linear $r$-graph such that
 the linear girth of it is greater than $\ell$,
and the number of edges in it is at least $\frac{1}{2r(r-1)}n^{1+ \frac{1}{\ell}}$, 
while we can not have the number of edges is $o(n^{1+ \frac{1}{\ell-1}})$ 
by taking $p=o(n^{- \frac{\ell-2}{\ell-1}})$ in this way.

Inspired by the ideas in addressing Erd\H{o}s conjecture for Steiner systems, especially the ones in~\cite{boh19,glock20},
we considered the constrained random greedy $r$-clique removal process  to
enumerate $|\textmd{Forb}_L(n,r,\ell)|$. 
As stated in~\cite{joos2025}, such random processes are easy to formulate, in many cases however,
a precise analysis is challenging. The central questions often concern structural 
properties that typically hold for the objects generated at termination.
For integers $r\geqslant 3$ and $\ell\geqslant 3$,
 the constrained process
does not just produce a linear $r$-graph but with linear girth greater than $\ell$,
 and the corresponding process is called  the \textit{random greedy high linear-girth $r$-clique removal process}
 or \textit{high linear-girth linear $r$-uniform hypergraph process}.
Compared with their analysis in~\cite{glock20,boh19} discussing sparse Steiner triple systems, 
the challenge 
arising in the analysis of random greedy high linear-girth $r$-clique removal process is that 
when each individual 
forbidden configuration $C_i^r$ with $3\leqslant i\leqslant \ell$ is considered,
it is more sensitive and stricter to establish 
 the trajectories and error functions of an ensemble
 of random variables which we track in  a higher uniformity,
such that the one-step
 changes, trend hypotheses and boundedness hypotheses of these tracked variables can 
 support the proof about the concentration of all these variables to the end.  
 Finally, we show the following theorem.
 \begin{theorem}
For every pair of fixed integers $r,\ell\geqslant 3$, consider the random 
greedy high linear-girth $r$-clique
removal process on vertex set $[n]$. Let $M$  be the number of edges in the
generated linear $r$-graph with linear girth greater than $\ell$ when the process terminates. 
With high probability, there exists some positive constant $\lambda=\lambda(\ell)$ 
 such that
\begin{align*}
M\geqslant n^{1+ \frac{1}{\ell-1}-\lambda \frac{\log\log n}{\log n}}.
\end{align*}
 \end{theorem} 
\noindent 
According to the proof of Theorem~1.1 in Section~4, we will show the term $O({\log\log n}/{\log n})$ 
 can not be removed from the exponent of
 the lower bound of
 $M$ and it is sufficient in our proof to show the one-step
 changes, trend hypotheses and boundedness hypotheses of these tracked variables
 satisfy the conditions in Freedman’s
martingale concentration inequality. Furthermore, we will show that applying the
random greedy process analysis to obtain a better lower bound of $M$ 
than $\Omega(n^{1+ \frac{1}{\ell-1}})$ is impossible. 
 We believe removing the term $O({\log\log n}/{\log n})$ in
 Theorem~1.1 entails some new substructures and analysis.

Combined with the known results that
$|\textmd{Forb}_L(n,r,3)|= 2^{o (n^{2})}$ 
 and 
 $|\textmd{Forb}_L(n,3,4)|= 2^{\Theta (n^{3/2})}$ extended from~\cite{balgoh17,erd86}, in this paper,
as a corollary of Theorem 1.1, it is natural to  obtain
that 
\begin{corollary} 
For every pair of fixed integer $r,\ell\geqslant 4$, or $r=3$ and $\ell\geqslant 5$,
there exists $\lambda=\lambda(\ell)>0$ such that as $n\rightarrow\infty$,  
\begin{align*}
\textmd{ex}_L(n,r,\ell)\geqslant n^{1+ \frac{1}{\ell-1}- \frac{\lambda\log\log n}{\log n}}\quad \text{and}\quad 
|\textmd{Forb}_L(n,r,\ell)|\geqslant 2^{n^{1+ \frac{1}{\ell-1}- \frac{\lambda\log\log n}{\log n}}}.
\end{align*}
\end{corollary}
\noindent

The remainder of the paper is structured as follows. 
Notation and  auxiliary
results used throughout the paper are presented in Section 2.
We define the random greedy high linear-girth $r$-clique
removal process in Section 3, introducing some key random variables of the process that we wish to track,  
estimating the corresponding expected trajectory and 
 choosing error function for each tracked variable.
 We formally prove the concentration of all these variables in Section 4,
 where the required one-step changes, trend hypotheses and boundedness hypotheses for 
these tracked variables are analyzed in higher uniformities.
    Two important claims to bound the overcounting and boundedness parameters
    are discussed in detail in  the final section.

\section{Preliminaries}

We adopt the standard asymptotic Landau notation and expressions from~\cite{boh19} throughout. 
All asymptotics in this paper are with respect to $n\to\infty$. 
For two positive-valued functions $a(n)$, $b(n)$ on the variable $n$,
write $a(n)=O(b(n))$ if there exists a constant $C>0$ such that for all
$n$, $|a(n)|\leqslant C\cdot |b(n)|$. Write $a(n)=o(b(n))$ or $a(n)\ll b(n)$ if
eventually $b(n)>0$ and $\lim_{n\rightarrow \infty}a(n)/b(n)$ exists and 
equals $0$; $a(n)=\omega(b(n))$ if eventually $a(n)>0$ and $b(n)=o(a(n))$;
$a(n)=b(n)(1+o(1))$ or $a(n)\sim b(n)$ means $\lim_{n\rightarrow \infty}a(n)/b(n)=1$.
Denote $a(n)=\Omega(b(n))$ if $b(n)=O(a(n))$.
Further, $a(n)=\Theta(b(n))$ if both $a(n)=O(b(n))$ and $b(n)=\Omega(a(n))$. 
Let $a=b\pm c$ be short for $a\in [b-c,b+c]$.
Equations containing $\pm$ are always to be interpreted from left to right,
i.e. $b_1\pm c_1=b_2\pm c_2$ means that $b_1-c_1\geqslant b_2-c_2$ and 
$b_1+c_1\leqslant b_2+c_2$. 
Moreover, given a set ${S}$ and integers $a, b\geqslant 0$,
we write ${{S}\choose b}$ for the collection of all $b$-sets of ${S}$;
in particular, ${{S}\choose b}=\emptyset$ and ${a\choose b}=0$ if $b>|{S}|$ and $b>a$;
$a\wedge b$ denotes the minimum of $a$ and $b$. 
All logarithms are natural. 
The floor and ceiling signs are omitted whenever they are not crucial.
 
Let $(\Omega,\mathcal{F},\mathbb{P})$ be an arbitrary probability space,
and $(\mathcal{F}_i)_{i\geqslant 0}$ be the filtration given by the evolution of
the process. For an event $\mathcal{A}$ or a random variable ${Z}$ in $(\Omega,\mathcal{F},\mathbb{P})$, 
let $\mathbb{P}[\mathcal{A}]$, $\mathbbm{1}_{\mathcal{A}}$ and $\bar{\mathcal{A}}$ denote the probability of
$\mathcal{A}$, the indicator function of $\mathcal{A}$ and the complement of $\mathcal{A}$; $\mathbb{E}[{Z}]$ 
and $\mathbb{V}[{Z}]$ denote the expectation and variance of ${Z}$.
Let ${X}(i)$ be a set which contains all elements with a 
certain characteristics at the step $i$ in the process. 
For a sequence of random variables $\{|{X}(i)|\}_{i\geqslant 0}$,
 let $\Delta {X}(i)=|{X}(i+1)|-|{X}(i)|$ denote the one-step change
of the random variable $|{X}(i)|$. The pair $\{|{X}(i)|,\mathcal{F}_i\}_{i\geqslant 0}$
 is then called a submartingale or a supermartingale
if $|{X}(i)|$ is $\mathcal{F}_i$-measurable  and $\mathbb{E}[\Delta {X}(i)|{\mathcal{F}}_i]\geqslant 0$
or $\mathbb{E}[\Delta {X}(i)|{\mathcal{F}}_i]\leqslant 0$ for all $i\geqslant 0$, respectively.
As usual, an event 
is said to occur  with high probability (\textit{w.h.p.}  for short), 
if the probability that it holds tends to 1 when $n\rightarrow\infty$.

Rescale the number of steps $i$
  to be $t=t_i= \frac{i}{n^2}$ such that  $t_{i+1}= t+ \frac{1}{n^2}$
  and every function of the discrete variable $i$ 
  can be viewed as  a continuous function
of $t$. It is sometimes more convenient to think of our random variables as depending
on a continuous variable $t$ rather than the discrete variable $i$.
We pass between these interpretations without comments.
The main tool in the proof of our main results is the differential equations method,
which applies a pseudo-random heuristic to divine the trajectories 
of a group of graph parameters that govern 
the evolution of the process. It
  plays a central role 
in the understanding of several other constrained random processes to produce 
interesting combinatorial objects~\cite{bennett15,bohman09,bohman15,boh19,glock20,glock2024,joos2025,kuhn16,rw92,rw97,tian23}.
The key in applying the method is to find a collection of random variables  
containing $|{X}(i)|$ whose one step change can be expressed in terms of 
other variables in the collection. These expressions form a system of difference 
equations and the solution to the initial value problem for the corresponding
system of differential equations gives a natural guess for the 
expected trajectories of the variables. The one-step change and boundedness hypotheses for
each tracked variable are analyzed, then the last step is to apply 
a martingale concentration inequality and a union bound to prove 
that all of the variables in the collection are concentrated
around their expected trajectories. 

Let the event $\mathcal{E}_{{X}}$ be of the form $|{X}(i)|= 
{x}(t)\pm \epsilon_{x}(t)$ for all $i\leqslant {M}$,
where ${x}(t)$ is the expected trajectory of $|{X}(i)|$, 
$\epsilon_{x}(t)$ is the error function
to control the deviation of $|{X}(i)|$ from ${x}(t)$. 
Let the stopping time $\tau$ be the smallest index $i$ such that any one of the random variables 
that we are tracking has strayed far from its expected trajectory
at or before the $i$th step,
 making it possible to establish the martingale condition.
We show that the event $\{\tau={M}\}$ holds by means of
$\{\tau={M}\}=\cap_{{X}\in\mathcal{I}}\mathcal{E}_{{X}}$, where $\mathcal{I}$ 
is the family of all key variables at every step and $| \mathcal{I}|$ is polynomial in $n$. 
We define a pair of sequences
of auxiliary random variables for each $|{X}(i)|$ as
\begin{align}
{X}^{\pm}(i)&=\pm\Bigl[|{X}(i)|-{x}(t)\Bigr]-\epsilon_{x}(t).
\end{align}
Note that the desired estimate $|{X}(i)|={x}(t)\pm \epsilon_{x}(t)$ 
follows if the estimate ${X}^{\pm}(i)\leqslant 0$ holds. 
To prove ${X}^{\pm}(i)\leqslant 0$ holds \textit{w.h.p.}, 
we first establish that the sequence ${X}^{\pm}(i)$ is a supermartingale, and
then provide bounds on the one-step change $\Delta {X}^{\pm}(i)={X}^{\pm}(i+1)-{X}^{\pm}(i)$
for each variable.
From~(2.1), for $0\leqslant i\leqslant {M}$, observe that 
\begin{align}
\Delta {X}^{\pm}(i)&=\pm \Delta {X}(i)\mp \Bigl[{x}\bigl(t+ \frac{1}{n^2}\bigr)-{x}(t)\Bigr]
- \Bigl[\epsilon_{x}\bigl(t+ \frac{1}{n^2}\bigr)-\epsilon_{x}(t)\Bigr]\notag\\
&=\pm \biggl[\Delta {X}(i)- \frac{{x}'(t)}{n^2}\biggr]- \frac{\epsilon_{x}'(t)}{n^2}+O\biggl(\sup_{\varsigma\in[t, t+ \frac{1}{n^2}]} \frac{|{x}''(\varsigma)|+|\epsilon_{x}''(\varsigma)|}{n^4}\biggr),
\end{align}
where the second equality is derived 
by applying Taylor's theorem with remainder in Lagrange form on  
${x}(t+ \frac{1}{n^2})-{x}(t)$ and $\epsilon_{x}(t+ \frac{1}{n^2})-\epsilon_{x}(t)$ until
the second order. Based on the estimated expression $x(t)$
and error function $\epsilon_{x}(t)$ into the equations~(2.1) and (2.2),
 we will show that the terms $\mathbb{E}[\Delta {X}(i)\,| \mathcal{F}_i]$ and
$ \frac{{x}{'}(t)}{n^2}$ cancel each other out; $ \frac{\epsilon_{x}'(t)}{n^2}$
dominates the remaining terms such that
$\mathbb{E}[\Delta {X}^\pm(i)\,|{\mathcal{F}}_i]\leqslant 0$ and 
$\mathbb{E}[|\Delta {X}^\pm(i)|\,|{\mathcal{F}}_i]=O( \frac{|\epsilon_{x}'(t)|}{n^2})$.
We also show that $|\Delta {X}^{\pm}(i)|\leqslant |\Delta {X}(i)|+ 
O(\max\{ \frac{|{x}'(t)|}{n^2}, \frac{|\epsilon_{x}'(t)|}{n^2}\})$ 
and the initial value ${X}^{\pm}(0)\leqslant- \frac{1}{2}\epsilon_{x}(0)$.
Applying  Freedman's martingale 
concentration inequality~\cite{freed75} (see also~\cite{warnke16}) and a union bound
of the events ${X}^{\pm}(i)\geqslant 0$ for all tracked variables at every step $i\leqslant M$, 
we prove 
that the designated variable ${X}^{\pm}(i)\geqslant 0$ 
from some $i\leqslant {M}$ has extremely low probability.
 For this to work, it is essential to strictly characterize 
$x(t)$ and $\epsilon_{x}(t)$ for each tracked variable,
 where
 $\epsilon_{x}(t)$  has a large enough growth rate throughout 
 the process, but $\epsilon_{x}(t)$ must not grow too fast, otherwise we would lose
 control of ${X}(i)$.

\section{High linear-girth $r$-clique removal process}

\subsection{Definitions and Variables}

\begin{definition}
For any given integer $i$ with $3\leqslant i\leqslant \ell$, 
define a Type-$i$ configuration to be a linear $r$-graph $L$ with ${e}_L=i$ 
which spans at most  ${v}_L\leqslant (r-1)\cdot i$ vertices, with the property that $L$ contains no 
subhypergraph $G\subsetneq L$ with $3\leqslant {e}_{G}< i$ edges on 
${v}_{G}\leqslant(r-1)\cdot {e}_{G}$ vertices. Let $\mathcal{L}=\mathcal{L}_\ell$ 
 denote the collection of all  Type-$i$ configurations $L$ with
$3\leqslant i\leqslant \ell$.
\end{definition} 
\begin{remark}It is observed that 
the structure of Type-$i$ configuration is a linear cycle $C_i^r$, 
because it is linear with linear girth exactly $\ell$, which
 implies that ${v}_L=(r-1)\cdot {e}_L$ for any 
$L\in\mathcal{L}$, and
${v}_{G}\geqslant (r-1)\cdot {e}_{G}+1$ for any $G\subsetneq L$ with $L\in\mathcal{L}$. 
For any fixed integer $\ell$,  we have $|\mathcal{L}|=O(1)$.
The definition here is helpful for analysing  
the probability of availability for each vertex set of an $r$-clique in~(3.8),
and the ideal expected values of these 
tracked variables in~(3.10), (3.13) and (3.14). The form in this definition
also facilitates the analysis of two groups
of auxiliary enumerative parameters in Section 5. 
\end{remark}

The \textit{random greedy high linear-girth $r$-clique removal process} starts with a complete graph
 on vertex set $[n]$, denoted by $\mathbb{G}(0)$,
 and $\mathbb{G}(i+1)$ is the remaining graph from $\mathbb{G}(i)$ by
 selecting one $r$-clique uniformly at random out of all \textit{available} $r$-cliques in $\mathbb{G}(i)$
 and deleting all its edges from the edge set ${E}(i)$ of $\mathbb{G}(i)$,
 where an \textit{available} $r$-clique in $\mathbb{G}(i)$ means that adding
 the vertex set of this $r$-clique as an edge
 does not produce an $r$-graph in $\mathcal{L}$  with the edges generated by
 the vertices of previously chosen available $r$-cliques.
 Let ${M}=\min\{i:\mathbb{G}(i) {\text{ is} \textit{ available } r{\text{-}\textit{clique} \text{ free}}}\}$.
 This process is equivalently viewed as creating
 a random linear $r$-graph with linear girth greater than $\ell$, that is, the 
 \textit{random greedy high linear-girth linear $r$-uniform hypergraph process}.
 Beginning with the empty $r$-graph $\mathbb{H}(0)$ on the same vertex set $[n]$ we 
sequentially set $\mathbb{H}({i+1})=\mathbb{H}(i)+\boldsymbol{e}_{i+1}$, where the added $r$-set $\boldsymbol{e}_{i+1}$
is the vertex set of the chosen available $r$-clique in $\mathbb{G}(i)$, and 
$\boldsymbol{e}_{i+1}$ is also called to be an available $r$-set or available edge for $\mathbb{H}(i)$. 
  The process terminates at a maximal linear $r$-graph $\mathbb{H}({M})$ with linear girth greater than $\ell$
  and  no available $r$-sets or available edges. At time $i$, graphs are defined with 
  respect to edge sets on the vertex set $[n]$, where
  both ${\mathbb{G}}(i)$ and $\mathbb{H}(i)$ depend on the outcomes of the process up
  to this point, and ${\mathbb{G}}(i)$ is the graph given by the pairs of vertices 
  that do not appear in any edge of $\mathbb{H}(i)$. 
It will be convenient to study the relation 
  between $\mathbb{G}(i)$ and $\mathbb{H}(i)$ via
\begin{align}
{E}(i)=\binom{[n]}{2}\Bigl\backslash \bigcup_{\boldsymbol{f}\in \mathbb{H}(i)}
\biggl\{\bigl\{x,y\bigr\}\in \binom{\boldsymbol{f}}{2}\biggr\}. 
\end{align}

For $2\leqslant m\leqslant r$, ${u}\in[n]$ and $\boldsymbol{f}_m=\{{u}_1,\cdots,{u}_m\}\in \binom{[n]}{m}$,
let ${N}_{{u}}(i)=\{{x}\in [n]:\{x,u\}\in {E}(i)\}$,
${N}_{\boldsymbol{f}_m}(i)=\cap_{j=1}^m {N}_{{u}_j}(i)$ and 
 ${K}_m(i)$ be the set of vertex sets of $m$-cliques in ${\mathbb{G}}(i)$. 
Due to the nature of the random greedy high linear-girth $r$-clique removal process, it should
come as no surprise that the most important variable for us to track is the number of 
available $r$-cliques in $\mathbb{G}(i)$. To this end, define
 ${Q}(i)$ to be the set of vertex sets of available $r$-cliques in ${\mathbb{G}}(i)$.
Our goal is to estimate the random variable $|{Q}(i)|$.
Note that $|{Q}(0)|= \binom{n}{r}$, and at termination
of the process $|{Q}({M})|=0$. 
Fix one $\boldsymbol{f}_m\in {K}_m(i)$, and define the set
${Y}_{\boldsymbol{f}_m}(i)$ to be 
\begin{align}
{Y}_{\boldsymbol{f}_m}(i)=\Bigl\{ \boldsymbol{f}'_{r-m}\in {{N}_{\boldsymbol{f}_m(i)}\choose r-m}\,:
\,\boldsymbol{f}_m \cup \boldsymbol{f}'_{r-m}\in {Q}(i)\Bigr\}.
\end{align}
Thus, $|{Y}_{\boldsymbol{f}_m}(i)|$ counts the number of available $r$-cliques in $\mathbb{G}(i)$ extended from 
the vertex set $\boldsymbol{f}_m\in {K}_m(i)$, and it is also called to be the 
\textit{available $(r-m)$-codegree} of $\boldsymbol{f}_m$.
Particularly, $|{Y}_{\boldsymbol{f}}(i)|=\mathbbm{1}_{\boldsymbol{f}}$, where $\mathbbm{1}_{\boldsymbol{f}}$ is 
the indicator random variable with $\mathbbm{1}_{\boldsymbol{f}}=1$ if $\boldsymbol{f}\in {Q}(i)$,
instead $\mathbbm{1}_{\boldsymbol{f}}=0$ otherwise. 

 We also track all possible routes 
to generate any copy of $L\in \mathcal{L}$ in $\mathbb{H}(i)$.
For any  $L\in \mathcal{L}$,  
let $\mathcal{F}_L$ be the set of all copies of $L$ in the complete $n$-vertex
$r$-graph. Given $L\in \mathcal{L}$, $\boldsymbol{f}\in {Q}(i)$ and $0\leqslant k\leqslant e_L-2$, 
define ${W}_{\boldsymbol{f},L,k}(i)$ to 
be the set of  $L'\in \mathcal{F}_L$ 
which is extended from the $r$-set $\boldsymbol{f}\in {Q}(i)$ such that $|L'\cap \mathbb{H}(i)|=k$
and $|L'\cap {Q}(i)|=e_L-k$, that is,
\begin{align}
{W}_{\boldsymbol{f},L,k}(i)&=\Bigl\{L'\in \mathcal{F}_L:\, \boldsymbol{f}\in L'\cap {Q}(i),\  
|L'\cap {Q}(i)|=e_L-k \text{\ and }|L'\cap \mathbb{H}(i)|=k\Bigr\}.
\end{align}

\subsection{Estimates on the variables}

A pseudo-random heuristic for divining the evolution of variables plays
a central role in the understanding of graph processes.
The following idea of estimating
on the key variables  has been applied in 
the environments of some interesting combinatorial objects~\cite{bennett15,bohman09,bohman15,boh19,glock20,joos2025,kuhn16, tian23}.
We extend the ideas in~\cite{glock20,boh19} to show the expected trajectories
for each tracked variable of the high linear-girth linear uniform hypergraph process
in a higher uniformity.

Suppose that the events that each available $r$-set, denoted by $\boldsymbol{f}$, is chosen 
to become an edge in $\mathbb{H}(i)$ are 
independent with probability  $\pi_i=\pi(t)$ without considering the forbidden configurations in $\mathcal{L}$
defined in Definition 3.1,
that is, $\mathbb{P}[\boldsymbol{f}\in \mathbb{H}(i)]=\pi_i$.
It implies that  $\mathbb{H}(i)$ 
is regarded as the binomial random $r$-graph
$\mathcal{H}_r(n,\pi_i)$, where every $r$-set appears independently with probability~$\pi_i$.
Recalling that $t=t_i= \frac{i}{n^2}$ and in view of $|\mathbb{H}(i)|=i\sim \binom{n}{r}\cdot r!t/n^{r-2}$, we have
\begin{align}
\pi_i\sim \frac{r!t}{n^{r-2}}.
\end{align}
From another viewpoint, 
 it means that the graph ${\mathbb{G}}(i)$ is the random graph whose variables are roughly
 the same as they are in the binomial random graph $\mathcal{G}(n,p_i)$,
 where every edge appears independently with probability~$p_i$ in $\mathcal{G}(n,p_i)$. 
Hence, we have\begin{align}
p_i&=p(t)=\mathbb{P}\bigl[\{x,y\}\in {E}(i)\bigr]=\mathbb{P}\Bigl[\bigcap_{\{\boldsymbol{f}\,|\{x,y\}\subseteq \boldsymbol{f}\}}
\Bigl\{\boldsymbol{f} \notin \mathbb{H}(i)\Bigr\}\Bigr]  
=1-\sum_{\{\boldsymbol{f}\,|\{x,y\}\subseteq \boldsymbol{f}\}}\mathbb{P}\bigl[
\boldsymbol{f} \in \mathbb{H}(i)\bigr]\notag\\
&\sim 1- \frac{n^{r-2}}{(r-2)!} \pi_i
\sim 1-r(r-1)t,
\end{align}
where the last approximate equality is derived from~(3.4).

Under the assumption that $\mathbb{H}(i)$ and ${\mathbb{G}}(i)$ respectively resemble $\mathcal{H}_r(n,\pi_i)$
and $\mathcal{G}(n,p_i)$, for any $\boldsymbol{f}_m\in {K}_m(i)$ with $2\leqslant m\leqslant r-1$, $\boldsymbol{f}\in {Q}(i)$, 
$L\in \mathcal{L}$ and $0\leqslant k\leqslant e_L-2$,
we try to estimate the expected trajectory of 
$|{Q}(i)|$ as $|{Q}(i)|={q}(t)\pm\epsilon_{q}(t)$, 
the likely values of  auxiliary random variables
$|{Y}_{\boldsymbol{f}_m}(i)|$ and $|{W}_{\boldsymbol{f},L,k}(i)|$ as 
$|{Y}_{\boldsymbol{f}_m}(i)|={y}_m(t)\pm\epsilon_{y_m}(t)$
and $|{W}_{\boldsymbol{f},L,k}(i)|={w}_{L,k}(t)\pm\epsilon_{w_{L,k}}(t)$.


Given
$\mathbb{H}(i)$, an $r$-set $\boldsymbol{f}\in {Q}(i)$
 means that
  $\{\{x,y\}\in \binom{\boldsymbol{f}}{2}\}\subseteq {E}(i)$
and $\mathbb{H}(i)+\boldsymbol{f}$ contains no $L'\in\cup_{L\in \mathcal{L}}\mathcal{F}_L$
satisfying $\boldsymbol{f}\in L'$,
then it follows that 
\begin{align}
\mathbb{E}[|{Q}(i)|]=\sum_{\boldsymbol{f}\in \binom{[n]}{r}}p_i^{ \binom{r}{2}}
\mathbb{P}\Bigl[\bigcap_{L\in \mathcal{L}}\bigl\{\mathbb{H}(i)+\boldsymbol{f}\text{ contains no }
L'\in {F}_L\text{ with }\boldsymbol{f}\in L'\bigr\}\Bigr].
\end{align}

Let 
\begin{align*}
\mathcal{N}_{\boldsymbol{f},L}=\bigl\{L'\in \mathcal{F}_L: \boldsymbol{f}\in L'\bigr\}
\end{align*}
denote the set of $L'\in \mathcal{F}_L$ satisfying $\boldsymbol{f}\in L'$ for 
any fixed $r$-set $\boldsymbol{f}\in \binom{[n]}{r}$ and  $L\in \mathcal{L}$.
We have
\begin{align}
|\mathcal{N}_{\boldsymbol{f},L}|&= \frac{r!e_L}{|\texttt{Aut}(L)|}n^{v_L-r}+O(n^{v_L-r-1})
\end{align}
because $|\mathcal{N}_{\boldsymbol{f},L}|\cdot \binom{n}{r} = |\mathcal{F}_L|\cdot e_L$
and $|\mathcal{F}_L|= {n(n-1)\cdots(n-v_L+1)}/{|\texttt{Aut}(L)|}=
 {n^{v_L}}/{|\texttt{Aut}(L)|}+O(n^{v_L-1})$,
where $\texttt{Aut}(L)$ is the automorphism group of $L$. 
The main term $ \frac{r!e_L}{|\texttt{Aut}(L)|} n^{v_L-r}$ in~(3.7) counts 
the number of $L'\in \mathcal{F}_L$ in $\mathcal{N}_{\boldsymbol{f},L}$
in which the intersection of all of them is exactly
$\boldsymbol{f}$.
Recall that $\mathbb{H}(i)$ 
is regarded as $\mathcal{H}_r(n,\pi_i)$ and
 $v_L=(r-1)\cdot e_L$ for any $L\in \mathcal{L}$ by Remark~3.2,
then we estimate 
\begin{align}
&\mathbb{P}\Bigl[\bigcap_{L\in \mathcal{L}}\bigl\{\mathbb{H}(i)+\boldsymbol{f}\text{ contains no }
L'\in \mathcal{F}_L\text{ with }\boldsymbol{f}\in L'\bigr\}\Bigr]\notag\\
&\sim
\prod_{L\in \mathcal{L}}\bigl(1-\pi_i^{e_L-1}\bigr)^{ \frac{r!e_L}{|\texttt{Aut}(L)|}n^{(r-1)e_L-r}}\notag\\
&\sim \xi_i,
\end{align}
where
\begin{align}
\xi_i&=\xi(t)=\exp\biggl[-\sum_{L\in \mathcal{L}} \frac{r!e_L}{|\texttt{Aut}(L)|}(r!t)^{e_L-1}n^{e_L-2}\biggr]
\end{align}
is the probability of availability for each vertex set of an
$r$-clique in $\mathbb{G}(i)$. 
As the equations shown in~(3.4),~(3.8) and~(3.9),
 we anticipate the expected value of $\mathbb{E}[|{Q}(i)|]$ in~(3.6) as
\begin{align}
\mathbb{E}[|{Q}(i)|]
&\sim \frac{n^r}{r!}p_i^{ \binom{r}{2}}\xi_i.
\end{align}

\begin{remark}
As the equation of $\xi_i$ shown in~(3.9), we also have
\begin{align}
 \xi_i'=-\xi_i\cdot \tilde{\xi}_i
 \end{align}
with 
\begin{align}
\tilde{\xi}_i=\tilde{\xi}(t)= \sum_{L\in \mathcal{L}} \frac{e_L(e_L-1)}{|\texttt{Aut}(L)|}(r!)^{e_L}(nt)^{e_L-2}.
\end{align}
Hence, $\tilde{\xi}_i=O((nt_M)^{\ell-2})$ with $t_M$ in~(3.26) and $|\mathcal{L}|=O(1)$ in Remark~3.2.
\end{remark}

Similarly,
for any $\boldsymbol{f}_m\in {K}_m(i)$ with $2\leqslant m\leqslant r-1$,
as the definitions of ${Y}_{\boldsymbol{f}_m}(i)$ 
in~(3.2), 
we estimate the expected value of
$|{Y}_{\boldsymbol{f}_m}(i)|$ to be 
\begin{align}
\mathbb{E}[|{Y}_{\boldsymbol{f}_m}(i)|]&\sim \frac{n^{r-m}}{(r-m)!}p_i^{\binom{r}{2}-\binom{m}{2}} \xi_i,
\end{align}
in which $\binom{n-m}{r-m}p_i^{\binom{r}{2}- \binom{m}{2}}\sim \frac{n^{r-m}}{(r-m)!}p_i^{\binom{r}{2}-\binom{m}{2}}$
counts the expected number of $r$-cliques extended from $\boldsymbol{f}_m$ in $\mathbb{G}(i)$;
$\xi_i$ is the probability of availability for each vertex set of an
$r$-clique in $\mathbb{G}(i)$ shown in~(3.9). 
For any $\boldsymbol{f}\in {Q}(i)$,  $L\in \mathcal{L}$ 
and $0\leqslant k\leqslant e_L-2$, as the definition of ${W}_{\boldsymbol{f},L,k}(i)$ in~(3.3),
we also approximate the expected value of $|{W}_{\boldsymbol{f},L,k}(i)|$ as
 \begin{align}
\mathbb{E}[|{W}_{\boldsymbol{f},L,k}(i)|]&\sim |\mathcal{N}_{\boldsymbol{f},L}|\cdot
\binom{e_L-1}{k}\cdot \pi_i^k \cdot (p_i^{\binom{r}{2}} \xi_i)^{e_L-k-1},
\end{align}
 where $|\mathcal{N}_{\boldsymbol{f},L}|\cdot\binom{e_L-1}{k}$ counts the ways to choose a $L'\in \mathcal{F}_L$ 
satisfying $\boldsymbol{f}\in L'$, and then choose $k$ edges in $L'-\boldsymbol{f}$;
$\pi_i^k (p_i^{\binom{r}{2}} \xi_i)^{e_L-k-1}$ represents 
the probability that these chosen $k$ edges are in $\mathbb{H}(i)$
and the remaining $e_L-k-1$ edges are 
available to $\mathbb{H}(i)$.

 Our main result below 
 shows that 
these random variables are concentrated around
the trajectories we have heuristically divined  in~(3.10), (3.13) and (3.14)
to verify our pseudo-random intuitions.

\begin{theorem}
Given any fixed $\mu>0$ and integers $r,\ell\geqslant 3$, there exist $\lambda>0$ and $\alpha\in (0,1)$
such that, for any ${\boldsymbol{f}}_m\in {K}_m(i)$ with $2\leqslant m\leqslant r-1$,
${\boldsymbol{f}}\in {Q}(i)$, $L\in \mathcal{L}$ and  $0\leqslant k\leqslant e_L-2$, 
with probability at least $1-n^{-\mu}$, 
\begin{align}
|{Q}(i)|&={q}(t)\pm\epsilon_{q}(t),\\
|{Y}_{{\boldsymbol{f}}_m}(i)|&={y}_m(t)\pm \epsilon_{{y}_m}(t),\\ 
|{W}_{{\boldsymbol{f}},L,k}(i)|&={w}_{L,k}(t)\pm \epsilon_{w_{L,k}}(t),
\end{align}
holding when $i$ satisfying $0\leqslant i\leqslant {M}$ 
with ${M}$ defined as
\begin{align}%
{M}=n^{1+ \frac{1}{\ell-1}-\lambda \frac{\log\log n}{\log n}}.
\end{align}
Here, for $t\in [0,t_M]$,
\begin{align}
{q}(t)&= \frac{n^r}{r!}p_i^{\binom{r}{2}}\xi_i,\\
{y}_m(t)&= \frac{n^{r-m}}{(r-m)!}p_i^{ \binom{r}{2}- \binom{m}{2}}\xi_i,\\
{w}_{L,k}(t)&= \frac{r!e_L}{|\texttt{Aut}(L)|}\binom{e_L-1}{k}(r!t)^{k}(p_i^{\binom{r}{2}}\xi_i)^{e_L-1-k}n^{(r-1)(e_L-k)+k-r},
\end{align}
and 
\begin{align}
\epsilon_{q}(t)&=\sigma_i n^{\alpha+r-1},\\
\epsilon_{y_m}(t)&=\sigma_i n^{\alpha+r-m-1},\\
\epsilon_{w_{L,k}}(t)&=\beta_k\sigma_i^{k+2} n^{\alpha+(r-1)(e_L-k)+k-r-1}t_{\textmd{M}}^k,
\end{align}
in which  
\begin{align}
\sigma_i=\sigma(t)=\log(n^\alpha+n^2t),
\end{align} $\beta_k$ is a constant depending on $k$ and
\begin{align}
t_{{M}}= \frac{{M}}{n^2}=n^{-1+ \frac{1}{\ell-1}-\lambda \frac{\log\log n}{\log n}}.
\end{align}
\end{theorem}

Theorem~3.4 is proved in Section 4 and we 
always tacitly assume $0\leqslant i\leqslant {M}$ with ${M}$ defined in~(3.18).
 It implies that the process produces
a linear $r$-graph of size at least ${M}$ with its linear girth greater than $\ell$.
Theorem~3.4 verifies Theorem~1.2 
with room to spare in the power of the logarithmic factor.
We make no attempts to optimize the constants $\lambda$, $\alpha$ and 
the coefficients $\beta_k$ for any $L\in \mathcal{L}$
 and  $0\leqslant k\leqslant e_L-2$. 
There are many choices of
them that can be balanced to satisfy certain inequalities. 
For example, for given positive integers $r$, $\ell$ and $k$,  we
choose these constants to satisfy the equations %
\begin{align}
\lambda> \frac{\ell}{\ell-1},\quad\alpha_0<\alpha<1 \quad\text{and}
\quad\beta_k= \Bigl(\frac{3\ell r!}{p_{{M}}^{ \binom{r}{2}}\xi_{{M}}}\Bigr)^{k},
\end{align}
where $\alpha_0= \frac{\ell-2}{\ell-1}$, $p_{{M}}$ and $\xi_{{M}}$ are constants 
when $i=M$ defined in~(3.5) and~(3.9). 
We will see that these choices are sufficient for our proof of Theorem 3.4 in next section.
 We do not replace them  with their actual values, which is for the interest of
understanding the role of these constants played in the calculations.

The equation of $|{Q}(i)|$ shown in~(3.15) follows 
from the equation of $|{Y}_{{\boldsymbol{f}}_2}(i)|$ in~(3.16) for any ${\boldsymbol{f}}_2\in {E}(i)$. 
Indeed, 
any $r$-set ${\boldsymbol{f}}\in {Q}(i)$ is counted
$ \binom{r}{2}$ times in $\sum_{{\boldsymbol{f}}_2\in {E}(i)}|{Y}_{{\boldsymbol{f}}_2}(i)|$ 
for any ${\boldsymbol{f}}_2\in {E}(i)$.
Since $|{E}(i)|=\binom{n}{2}- \binom{r}{2}i$ and $p_i=1-r(r-1)t$ in~(3.5),
it follows that  
\begin{align*}
|{E}(i)|&= \frac{n^2p_i}{2}\cdot (1- \frac{1}{np_i})
\end{align*} and
\begin{align}
|{Q}(i)|&=  \frac {1}{ \binom{r}{2}}\sum_{{\boldsymbol{f}}_2\in {E}(i)}|{Y}_{{\boldsymbol{f}}_2}(i)|\notag\\
&=\frac{n^2p_i}{[r]_2}\Bigl(1- \frac{1}{np_i}\Bigr)
 \bigl({y}_2(t)\pm \epsilon_{y_2}(t)\bigr) \notag\\
&={q}(t)\pm \epsilon_{q}(t), 
\end{align}
where $[r]_i=\prod_{j=0}^{i-1}(r-j)$ denotes the falling factorial, 
$ \frac{n^2p_i}{[r]_2}\cdot{y}_2(t)={q}(t)$ by~(3.19) and~(3.20), 
$ \frac{n^2p_i}{[r]_2}\cdot\epsilon_{y_2}(t)<\epsilon_{q}(t)$ and $ \frac{n}{[r]_2}\cdot({y}_2(t)\pm \epsilon_{y_2}(t))=
o(\epsilon_{q}(t))$ by (3.22) and 
(3.23). 
It thus suffices to establish the bounds in Theorem~3.4 of 
$|{Y}_{\boldsymbol{f}_m}(i)|$ for any $\boldsymbol{f}_m\in {K}_m(i)$ 
with $2\leqslant m\leqslant r-1$ and of $|{W}_{\boldsymbol{f},L,k}(i)|$
for any $\boldsymbol{f}\in {Q}(i)$, 
$L\in \mathcal{L}$ and $0\leqslant k\leqslant e_L-2$ in next section.

\vskip 0.3cm 
\begin{remark}
According to (3.5) and (3.9), note that the decreasing functions $p_i$
and $\xi_i$ satisfy
 $p_i,\xi_i=\Theta(1)$ when $t\in [0, t_{{M}}]$ 
 with $t_{{M}}$ defined in~(3.26). 
 Unfortunately, it is not allowed to choose $t_M=n^{-1+ \frac{1}{\ell-1}}$
 in the proof of our main results in Section~4 because
 we can not remove $\sigma_i$ or change the power of $\sigma_i$ in 
(3.22)-(3.24), where the factor $\sigma_i$ and its power in (3.23) and (3.24)
are sufficient to show the sequence of 
 each tracked variable is a supermartingales 
 based on the analysis and calculations in Section~4.2.1 and 4.2.2,
 and the factor $\sigma_i$ is sufficient to apply 
 Lemma~4.3.1 in Section~4.3.
\end{remark}

\vskip 0.3cm
\begin{remark}
It is impossible to apply the random greedy analysis to show 
a better lower bound for the step $i$ of the process
than $\Omega(n^{1+ \frac{1}{\ell-1}})$ for any pair of fixed integers $r,\ell\geqslant 3$. 
In fact, $\xi_i\rightarrow 0$ as  $t= 
n^{-1+ \frac{1}{\ell-1}+\varepsilon}$ for any arbitrarily small $\varepsilon>0$. 
Hence, $\mathbb{E}[|{Q}(i)|]$ in~(3.10),
$\mathbb{E}[|{Y}_{\boldsymbol{f}_m}(i)|]$ in~(3.13) and 
$\mathbb{E}[|{W}_{\boldsymbol{f},L,k}(i)|]$ in~(3.14)
all tend to zero such that the analysis here does not work.
\end{remark}

\vskip 0.3cm    
\begin{remark} 
Set $\hat{\epsilon}_{q}(t)=\epsilon_{q}(t)/{q}(t)$ and 
$\hat{\epsilon}_{y_m}(t)=\epsilon_{y_m}(t)/{y}_m(t)$ for any $2\leqslant m\leqslant r-1$. 
We have 
\begin{align}
\hat{\epsilon}_{q}(t)=\hat{\epsilon}_{y_m}(t)=
O(\sigma_i n^{\alpha-1})
\end{align} as the equations shown in Theorem~3.4.
\end{remark}

\vskip 0.3cm    
 \begin{remark}
In the following sections, in order to express the formulas concise and clear, 
unless otherwise specified, we often suppress
the dependence on $t$
 in notation $q(t)$, $y_m(t)$, $w_{L,k}(t)$, $\epsilon_q(t)$, $\epsilon_{y_m}(t)$
 and $\epsilon_{w_{L,k}}(t)$  in Section~4. 
\end{remark}

\section{Analysis of the high linear-girth process}

For any integer $0\leqslant i< {M}$,
let $\mathcal{G}(i)$ denote the event that the
estimates in Theorem~3.4 hold for any
 $\boldsymbol{f}_m\in {K}_m(j)$ with $2\leqslant m\leqslant r-1$,
 $\boldsymbol{f}\in {Q}(j)$, 
$L\in \mathcal{L}$ and $0\leqslant k\leqslant e_L-2$
when $0\leqslant j\leqslant i$.
The initial conditions in $\mathbb{G}(0)$ satisfy the equations in~(3.15)-(3.17), 
which implies that the event $\mathcal{G}(0)$ always holds.
For any fixed constant $\mu>0$, assume that the event $\mathcal{G}(i)$
holds with probability  at least $1-n^{-\mu}$. 
Following the approach in~\cite{glock20,boh19} to consider sparse Steiner triple systems,  
we analyze the random greedy high linear-girth linear $r$-uniform hypergraph process in this section
and finally show that the event $\mathcal{G}({{M}})$ 
 holds with high probability to complete the proof of Theorem~3.4. 

Compared with the analysis of triple systems, 
conditioned on the event $\mathcal{G}(i)$
holds, it is challenging to establish 
$\mathbb{E}[\Delta {Y}_{\boldsymbol{f}_m}^\pm(i)\,|{\mathcal{F}}_i]$ and 
$\mathbb{E}[\Delta {W}_{\boldsymbol{f},L,k}^\pm(i)\,|{\mathcal{F}}_i]$,
 $|\Delta {Y}_{\boldsymbol{f}_m}^{\pm}(i)|$ and $|\Delta {W}_{\boldsymbol{f},L,k}^{\pm}(i)|$
 in higher uniformities such that they comply with martingale 
concentration inequalities.
 We bound the probability
of the event 
that the tracked variable $ {Y}^{\pm}_{\boldsymbol{f}_m}(i)\geqslant 0$ 
for some $\boldsymbol{f}_m\in {K}_m(i)$ 
with $2\leqslant m\leqslant r-1$,   or $ {W}^{\pm}_{\boldsymbol{f},L,k}(i)\geqslant 0$
for some $\boldsymbol{f}\in {Q}(i)$, 
$L\in \mathcal{L}$ and $0\leqslant k\leqslant e_L-2$ such that
an application of the union bound for  all variables over all steps $i\leqslant M$ shows
that the probability of the occurrence of the event $\mathcal{G}(i+1)$
is extremely high.

For technical reasons we regulate $\Delta  {Y}_{\boldsymbol{f}_m}(i)=0$ as soon as
$\boldsymbol{f}_m\notin {K}_m(i+1)$ and $\Delta {W}_{\boldsymbol{f},L,k}(i)=0$
as soon as $\boldsymbol{f}\notin {Q}(i+1)$ since the relevant
structural constraints are violated.
In the following, in order to analyze $\Delta {Y}_{\boldsymbol{f}_m}(i)$
(resp. $\Delta {Y}_{\boldsymbol{f}_m}^{\pm}(i)$) and $\Delta {W}_{\boldsymbol{f},L,k}(i)$
(resp. $\Delta {W}_{\boldsymbol{f},L,k}^{\pm}(i)$),
we always assume $\boldsymbol{f}_m\in {K}_m(i+1)$
and $\boldsymbol{f}\in {Q}(i+1)$. 

\subsection{The one-step (expected) changes of 
$|{Y}_{\boldsymbol{f}_{m}}(i)|$ 
and $|{W}_{\boldsymbol{f},L,k}(i)|$}%

\subsubsection{$\mathbb{E}[\Delta {Y}_{\boldsymbol{f}_{m}}(i)\,| \mathcal{F}_i]$ 
}

Given one $\boldsymbol{f}_{m}\in {K}_{m}(i)$ with $2\leqslant m\leqslant r-1$,
 consider $\mathbb{E}[\Delta {Y}_{\boldsymbol{f}_{m}}(i)\,| \mathcal{F}_i]$.
Choosing one element from ${Y}_{\boldsymbol{f}_{m}}(i)$, denoted  as
 $\boldsymbol{f}_{m}^{c}$, let $\boldsymbol{f}=\boldsymbol{f}_{m}\cup \boldsymbol{f}_{m}^{c}$.
We enumerate all possibilities to remove  an
 $r$-set from ${Q}(i)$ such  that 
$\boldsymbol{f}_{m}^{c}$ is removed from ${Y}_{\boldsymbol{f}_{m}}(i)$. 
The first way is that the chosen $r$-set makes $\boldsymbol{f}$ not the vertex set 
of an $r$-clique. Let ${Q}_{\boldsymbol{f}_{m}^{c}}(i)$
denote the collection of such available $r$-sets in ${Q}(i)$.
The second way is that
the chosen $r$-set makes
$\boldsymbol{f}$ still the vertex set of an $r$-clique, while
$\boldsymbol{f}$ is not available anymore. 
Note that a bijection is established between such $r$-sets and 
the elements in $\bigcup_{L\in \mathcal{L}}{W}_{\boldsymbol{f},L,e_L-2}(i)$,
where no $r$-sets are associated with these two ways because $L\in \mathcal{L}$ is linear.
 Combining these two ways together,
it follows that
 \begin{align}
\mathbb{E}[\Delta {Y}_{\boldsymbol{f}_{m}}(i)\,| \mathcal{F}_i]
&=-\sum_{{\boldsymbol{f}_{m}^{c}}\in {Y}_{\boldsymbol{f}_{m}}(i)}
\frac{|{Q}_{\boldsymbol{f}_{m}^{c}}(i)|+|\bigcup_{L\in \mathcal{L}}{W}_{\boldsymbol{f},L,e_L-2}(i)|}{|{Q}(i)|}.
\end{align}

For $\boldsymbol{f}$, $\boldsymbol{g}\in {Q}(i)$ with  $\boldsymbol{f}\neq \boldsymbol{g}$,
since $\mathbb{H}(i)+\{\boldsymbol{f},\boldsymbol{g}\}$ may 
contain multiple linear $r$-graphs $L\in\mathcal{L}$, it might not be true to estimate 
$|\bigcup_{L\in \mathcal{L}}{W}_{\boldsymbol{f},L,e_L-2}(i)|$ 
as $\sum_{L\in \mathcal{L}}|{W}_{\boldsymbol{f},L,e_L-2}(i)|$.
To overcome this kind of possibility for overcounting, for any $\boldsymbol{f}\in {Q}(i)$, 
define $\Upsilon_{\boldsymbol{f}}(i)$ as
\begin{align}
\Upsilon_{\boldsymbol{f}}(i)&=\Bigl\{\boldsymbol{g}\in {Q}(i)\backslash\{\boldsymbol{f}\}: 
\mathbb{H}(i)+\{\boldsymbol{f},\boldsymbol{g}\}\text{\ contains multiple }L\in \mathcal{L}\text{ containing }\{\boldsymbol{f},\boldsymbol{g}\}\Bigr\},
\end{align}
which means that the selection of $\boldsymbol{g}\in {Q}(i)\backslash\{\boldsymbol{f}\}$
into $\mathbb{H}(i)$ makes $\boldsymbol{f}$ unavailable more than one time.
Rearrange the equation in~(4.1) to be
\begin{align}
\mathbb{E}[\Delta {Y}_{\boldsymbol{f}_{m}}(i)\,| \mathcal{F}_i]
&=-\sum_{{\boldsymbol{f}_{m}^{c}}\in {Y}_{\boldsymbol{f}_{m}}(i)}
\frac{|{Q}_{\boldsymbol{f}_{m}^{c}}(i)|+\sum_{L\in \mathcal{L}}|{W}_{\boldsymbol{f},L,e_L-2}(i)|+O(|\Upsilon_{\boldsymbol{f}}(i)|)}{|{Q}(i)|}.
\end{align}

We  analyze the
enumerative formula of $|{Q}_{\boldsymbol{f}_{m}^{c}}(i)|$ in~(4.3). Recall that ${Q}_{\boldsymbol{f}_{m}^{c}}(i)$
is the collection of $r$-sets in ${Q}(i)$ such that its choice into $\mathbb{H}(i)$
as an edge makes $\boldsymbol{f}=\boldsymbol{f}_{m}\cup \boldsymbol{f}_{m}^{c}$ not the vertex set 
of an $r$-clique in $\mathbb{G}(i+1)$. 

\vskip 0.3cm 
\noindent\textbf{Claim 4.1.1}\ 
\begin{align*}
|{Q}_{\boldsymbol{f}_{m}^{c}}(i)|
=\biggl[{r\choose 2}-{m\choose 2}\biggr]\cdot\bigl( {y}_2
\pm \epsilon_{y_2}\bigr)+O\bigl(n^{r-3}\bigr).
\end{align*}

\begin{proof}[Proof of Claim 4.1.1]\ 
Let $\bar{\boldsymbol{f}}_s\subseteq \boldsymbol{f}$ with 
$|\bar{\boldsymbol{f}}_s|=s$. 
Denote ${Q}_{\boldsymbol{f}_{m}^{c}}^{\bar{\boldsymbol{f}}_s}(i)\subseteq {Q}_{\boldsymbol{f}_{m}^{c}}(i)$
as the collection of available $r$-sets 
such that every $r$-set in ${Q}_{\boldsymbol{f}_{m}^{c}}^{\bar{\boldsymbol{f}}_s}(i)$
exactly contains the vertices $\bar{\boldsymbol{f}}_s$ in $\boldsymbol{f}$.
Since $\boldsymbol{f}_{m}\in {K}_m(i+1)$, to ensure 
that the removal of the edges of any one of these corresponding $r$-cliques 
in ${Q}_{\boldsymbol{f}_{m}^{c}}^{\bar{\boldsymbol{f}}_s}(i)$
results in $\boldsymbol{f}$  not the vertex set of an $r$-clique, 
we have $2\leqslant s\leqslant r$ and $|\bar{\boldsymbol{f}}_s\cap \boldsymbol{f}_{m}|\leqslant 1$.

Let ${\boldsymbol{f}_{m}\choose j}\oplus{\boldsymbol{f}_{m}^{c}\choose s-j}$
denote the collection of $s$-sets consisting of
the union of $j$ vertices in $\boldsymbol{f}_{m}$ and $s-j$ vertices in $\boldsymbol{f}_{m}^{c}$. 
We have
$\bar{\boldsymbol{f}}_s\in \cup_{j=0}^1{\boldsymbol{f}_{m}\choose j}\oplus{\boldsymbol{f}_{m}^{c}\choose s-j}$
and $|{Q}_{\boldsymbol{f}_{m}^{c}}(i)|$ is decomposed into
\begin{align}
|{Q}_{\boldsymbol{f}_{m}^{c}}(i)|
=\sum_{s=2}^{r-m}
\sum_{\bar{\boldsymbol{f}}_s\in {{\boldsymbol{f}_{m}^{c}\choose s}}}|{Q}_{\boldsymbol{f}_{m}^{c}}^{\bar{\boldsymbol{f}}_s}(i)|+
\sum_{s=2}^{r-m+1}
\sum_{\bar{\boldsymbol{f}}_s\in {{\boldsymbol{f}_{m}\choose 1}\oplus{\boldsymbol{f}_{m}^{c}\choose s-1}}}
|{Q}_{\boldsymbol{f}_{m}^{c}}^{\bar{\boldsymbol{f}}_s}(i)|.
\end{align}
Apply inclusion-exclusion to obtain $|{Q}_{\boldsymbol{f}_{m}^{c}}^{\bar{\boldsymbol{f}}_s}(i)|$ as
\begin{align*}
|{Q}_{\boldsymbol{f}_{m}^{c}}^{\bar{\boldsymbol{f}}_s}(i)|
&=|{Y}_{\bar{\boldsymbol{f}}_s}(i)|-\sum_{\boldsymbol{t}_1\in{\boldsymbol{f}
\backslash \bar{\boldsymbol{f}}_s\choose 1}}|{Y}_{\bar{\boldsymbol{f}}_s\cup \boldsymbol{t}_1}(i)|+\dots+ 
\sum_{\boldsymbol{t}_{r-s}\in{\boldsymbol{f}
\backslash \bar{\boldsymbol{f}}_s\choose r-s}}(-1)^{r-s}|{Y}_{\bar{\boldsymbol{f}}_s\cup \boldsymbol{t}_{r-s}}(i)|\\
&=\sum_{j=0}^{r-s}\sum_{\boldsymbol{t}_j\in {\boldsymbol{f}
\backslash \bar{\boldsymbol{f}}_s\choose j}}(-1)^j|{Y}_{\bar{\boldsymbol{f}}_s\cup \boldsymbol{t}_j}(i)|.
\end{align*}
 Combining with the equation in (4.4), it follows that
\begin{align}
|{Q}_{\boldsymbol{f}_{m}^{c}}(i)|
&=\sum_{s=2}^{r-m}\sum_{j=0}^{r-s}\sum_{\bar{\boldsymbol{f}}_s\in {{\boldsymbol{f}_{m}^{c}\choose s}}}
\sum_{\boldsymbol{t}_j\in {\boldsymbol{f}
\backslash \bar{\boldsymbol{f}}_s\choose j}}(-1)^j|{Y}_{\bar{\boldsymbol{f}}_s\cup \boldsymbol{t}_j}(i)|\notag\\
&\quad+\sum_{s=2}^{r-m+1}\sum_{j=0}^{r-s}\sum_{\bar{\boldsymbol{f}}_s\in {{\boldsymbol{f}_{m}\choose 1}
\oplus{{\boldsymbol{f}_{m}^{c}\choose s-1}}}}\sum_{\boldsymbol{t}_j\in {\boldsymbol{f}
\backslash \bar{\boldsymbol{f}}_s\choose j}}(-1)^j|{Y}_{\bar{\boldsymbol{f}}_s\cup \boldsymbol{t}_j}(i)|.
\end{align}

As to the equation of $|{Y}_{\boldsymbol{f}_m}(i)|$ for any $2\leqslant m\leqslant r-1$
in (3.16), $|{Y}_{\bar{\boldsymbol{f}}_s\cup \boldsymbol{t}_j}(i)|$ with $(s,j)=(2,0)$
dominates the sum on the right side of (4.5), then 
 $|{Q}_{\boldsymbol{f}_{m}^{c}}(i)|$ satisfies
\begin{align*}
|{Q}_{\boldsymbol{f}_{m}^{c}}(i)|
&=\biggl[{r-m\choose 2}+m(r-m)\biggl]\cdot\bigl( {y}_2
\pm \epsilon_{y_2}\bigr)+O\bigl(n^{r-3}\bigr),
\end{align*}
where 
${r-m\choose 2}$ enumerates $|{Y}_{\bar{\boldsymbol{f}}_s}(i)|$ when $\bar{\boldsymbol{f}}_s\in {
{\boldsymbol{f}_{m}^{c}\choose 2}}$,
$m(r-m)$ enumerates $|{Y}_{\bar{\boldsymbol{f}}_s}(i)|$  when $\bar{\boldsymbol{f}}_s\in {{\boldsymbol{f}_{m}\choose 1}
\oplus{\boldsymbol{f}_{m}^{c}\choose 1}}$, and 
$O(n^{r-3})$ is derived from the terms when $(s,j)\neq(2,0)$ in (4.5).
Note that ${r-m\choose 2}+m(r-m)={r\choose 2}-{m\choose 2}$, we complete the proof
of Claim 4.1.1.
\end{proof}

\subsubsection{$\mathbb{E}[\Delta {W}_{\boldsymbol{f},L,k}(i)\,|\mathcal{F}_i]$ 
}

Given $\boldsymbol{f}\in {Q}(i)$, $L\in \mathcal{L}$ and $0\leqslant k\leqslant e_L-2$,
we now consider $\mathbb{E}[\Delta {W}_{\boldsymbol{f},L,k}(i)\,|\mathcal{F}_i]$.
We consider all possibilities to make $|{W}_{\boldsymbol{f},L,k}(i)|$ changed by
adding an available $r$-set into $\mathbb{H}(i)$. 

Choose one $L'\in {W}_{\boldsymbol{f},L,k}(i)$, we firstly 
enumerate ways to remove $L'$ from ${W}_{\boldsymbol{f},L,k}(i)$ 
by deleting an $r$-set in ${Q}(i)\backslash \{\boldsymbol{f}\}$, 
which implies that some $\boldsymbol{g}\in (L'\cap {Q}(i))\backslash \{\boldsymbol{f}\}$ becomes unavailable.
Given one $\boldsymbol{g}\in (L'\cap {Q}(i))\backslash \{\boldsymbol{f}\}$,
similarly, the ways to 
make $\boldsymbol{g}$ unavailable are in two cases. %
One case is that $\boldsymbol{g}$ is not the vertex set 
of an $r$-clique. 
Let ${Q}_{\boldsymbol{g}}(i)$ 
denote the collection of such available $r$-sets in ${Q}(i)$.
 The other case is that the chosen $r$-set makes 
$\boldsymbol{g}$ still the vertex set of an $r$-clique, while 
$\boldsymbol{g}$ is unavailable. 
The number of such $r$-sets is counted as $\sum_{K\in \mathcal{L}}
|{W}_{\boldsymbol{g},K,e_K-2}(i)|+O(|\Upsilon_{\boldsymbol{g}}(i)|)$
as the analysis of $\mathbb{E}[\Delta {Y}_{\boldsymbol{f}_{m}}(i)\,| \mathcal{F}_i]$ in
Subsection~4.1.1, where $\Upsilon_{\boldsymbol{g}}(i)$ is defined in~(4.2).  
Define $\Psi_{L'}(i)$ as another
overcounting collection of $r$-sets to be 
\begin{align}
\Psi_{L'}(i)&=\Bigl\{\boldsymbol{h}\in {Q}(i)\backslash\{\boldsymbol{f}\}: 
\text{two }r\text{-sets in }(L'\cap{Q}(i))\backslash \{\boldsymbol{h}\}\notag\\
&\qquad\qquad\qquad\qquad\qquad\qquad\qquad\qquad\qquad\text{ are  unavailable in }
\mathbb{H}(i)+\{\boldsymbol{h}\}\Bigr\},
\end{align}
which implies that the selection of $\boldsymbol{h}\in {Q}(i)\backslash\{\boldsymbol{f}\}$ 
into $\mathbb{H}(i)$ 
makes two different $r$-sets 
in $(L'\cap{Q}(i))\backslash \{\boldsymbol{h}\}$ unavailable. 
Hence, for any given $L'\in {W}_{\boldsymbol{f},L,k}(i)$,
the number of ways to remove $L'$ from ${W}_{\boldsymbol{f},L,k}(i)$ is counted by
\begin{align}
\sum_{\boldsymbol{g}\in (L'\cap {Q}(i))\backslash\{\boldsymbol{f}\}}
\Bigl(|{Q}_{\boldsymbol{g}}(i)|+\sum_{K\in \mathcal{L}}|{W}_{\boldsymbol{g},K,e_K-2}(i)|
+O(|\Upsilon_{\boldsymbol{g}}(i)|)\Bigr)+O\bigl(|\Psi_{L'}(i)|\bigr).
\end{align}

Secondly, 
for $k\geqslant 1$, the ways to make one $L'\in {W}_{\boldsymbol{f},L,k-1}(i)$ move into 
${W}_{\boldsymbol{f},L,k}(i+1)$ 
are counted by 
\begin{align}
\sum_{L'\in {W}_{\boldsymbol{f},L,k-1}(i)}(e_L-k)+O\bigl(|\Lambda_{\boldsymbol{f},L,k-1}(i)|\bigr),
\end{align}
where $\Lambda_{\boldsymbol{f},L,k-1}(i)$ is the last overcounting 
collection  of  $r$-sets to be
\begin{align}
\Lambda_{\boldsymbol{f},L,k-1}(i)&=\Bigl\{\boldsymbol{h}\in L'\cap {Q}(i)\backslash\{\boldsymbol{f}\}
:
\text{an }r\text{-set in }(L'\cap{Q}(i))\backslash \{\boldsymbol{h}\}\notag\\
&\qquad\qquad\qquad\text{ is  unavailable in }
\mathbb{H}(i)+\{\boldsymbol{h}\},\text{ where }L'\in {W}_{\boldsymbol{f},L,k-1}(i)\Bigr\},
\end{align}
that is, the selection of $\boldsymbol{h}\in L'\cap {Q}(i)\backslash\{\boldsymbol{f}\}$
for some $L'\in {W}_{\boldsymbol{f},L,k-1}(i)$
 into $\mathbb{H}(i)$ as an edge makes another $r$-set 
of $(L'\cap {Q}(i))\backslash\{\boldsymbol{h}\}$ unavailable. 

Considering all these effects shown in~(4.7) and~(4.8) together, 
it follows that
\begin{align}
&\mathbb{E}[\Delta {W}_{\boldsymbol{f},L,k}(i)\,|\mathcal{F}_i]\notag\\&=- \frac{1}{|{Q}(i)|}\sum_{L'\in {W}_{\boldsymbol{f},L,k}(i)}
\biggl[\sum_{\boldsymbol{g}\in (L'\cap {Q}(i))\backslash\{\boldsymbol{f}\}}
\Bigl(|{Q}_{\boldsymbol{g}}(i)|+\sum_{K\in \mathcal{L}}|{W}_{\boldsymbol{g},K,e_K-2}(i)|
+O(|\Upsilon_{\boldsymbol{g}}(i)|)\Bigr)+O\bigl(|\Psi_{L'}(i)|\bigr)\biggr]\notag\\&
\qquad+ \frac{\mathbbm{1}_{\{k\geqslant 1\}}}{|{Q}(i)|}\biggl[\sum_{L'\in {W}_{\boldsymbol{f},L,k-1}(i)}(e_L-k)+O\bigl(|\Lambda_{\boldsymbol{f},L,k-1}(i)|\bigr)\biggr]. 
\end{align}

We analyze the enumerative formula of $|{Q}_{\boldsymbol{g}}(i)|$ in~(4.10). 
Recall that ${Q}_{\boldsymbol{g}}(i)$
is the collection of $r$-sets in ${Q}(i)$
such that its selection into $\mathbb{H}(i)$ as an edge 
makes $\boldsymbol{g}$ not the vertex set 
of an $r$-clique in $\mathbb{G}(i+1)$, 
where $\boldsymbol{g}\in (L'\cap {Q}(i))\backslash\{\boldsymbol{f}\}$
for some $L'\in {W}_{\boldsymbol{f},L,k}(i)$.

\vskip 0.3cm
\noindent\textbf{Claim 4.1.2}\  
\begin{align*}
|{Q}_{\boldsymbol{g}}(i)|
={r\choose 2}\cdot\bigl( {y}_2
\pm \epsilon_{y_2}\bigr)+O\bigl(n^{r-3}\bigr).
\end{align*}

\begin{proof}[Proof of Claim 4.1.2] 
Let $\bar{\boldsymbol{g}}_s\subseteq \boldsymbol{g}$ with $|\bar{\boldsymbol{g}}_s|=s$
for an integer $s$ satisfying $2\leqslant s\leqslant r$. 
Denote ${Q}_{\boldsymbol{g}}^{\bar{\boldsymbol{g}}_s}(i)\subseteq {Q}_{\boldsymbol{g}}(i)$
to be the collection of available $r$-sets such that every $r$-set in 
${Q}_{\boldsymbol{g}}^{\bar{\boldsymbol{g}}_s}(i)$
exactly contains 
 the vertices $\bar{\boldsymbol{g}}_s$ in $\boldsymbol{g}$.
 Hence, we rewrite
$|{Q}_{\boldsymbol{g}}(i)|$ into 
\begin{align}
|{Q}_{\boldsymbol{g}}(i)|=\sum_{s=2}^{r}
\sum_{\bar{\boldsymbol{g}}_s\in {\boldsymbol{g}\choose s}}|{Q}_{\boldsymbol{g}}^{\bar{\boldsymbol{g}}_s}(i)|.
\end{align}
We also apply the inclusion-exclusion to obtain $|{Q}_{\boldsymbol{g}}^{\bar{\boldsymbol{g}}_s}(i)|$ as
\begin{align*}
|{Q}_{\boldsymbol{g}}^{\bar{\boldsymbol{g}}_s}(i)|
&=|{Y}_{\bar{\boldsymbol{g}}_s}(i)|-\sum_{\boldsymbol{t}_1\in{\boldsymbol{g}
\backslash \bar{\boldsymbol{g}}_s\choose 1}}|{Y}_{\bar{\boldsymbol{g}}_s\cup \boldsymbol{t}_1}(i)|+\dots+
\sum_{\boldsymbol{t}_{r-s}\in{\boldsymbol{g}\backslash \bar{\boldsymbol{g}}_s\choose r-s}}(-1)^{r-s}|{Y}_{\bar{\boldsymbol{g}}_s\cup \boldsymbol{t}_{r-s}}(i)|\\
&=\sum_{j=0}^{r-s}\sum_{\boldsymbol{t}_j\in{\boldsymbol{g}\backslash \bar{\boldsymbol{g}}_s\choose j}}(-1)^j|{Y}_{\bar{\boldsymbol{g}}_s\cup \boldsymbol{t}_j}(i)|.
\end{align*}
 Combining with the equation in (4.11), it follows that
\begin{align*}
|{Q}_{\boldsymbol{g}}(i)|
&=\sum_{s=2}^{r}\sum_{j=0}^{r-s}\sum_{\bar{\boldsymbol{g}}_s\in {\boldsymbol{g}\choose s}}
\sum_{\boldsymbol{t}_j\in{\boldsymbol{g}\backslash \bar{\boldsymbol{g}}_s\choose j}}(-1)^j|{Y}_{\bar{\boldsymbol{g}}_s\cup \boldsymbol{t}_j}(i)|,
\end{align*}
in which
\begin{align*}
\sum_{\bar{\boldsymbol{g}}_s\in {\boldsymbol{g}\choose s}}
\sum_{\boldsymbol{t}_j\in{\boldsymbol{g}\backslash \bar{\boldsymbol{g}}_s\choose j}}|{Y}_{\bar{\boldsymbol{g}}_s\cup \boldsymbol{t}_j}(i)|
&=\sum_{\bar{\boldsymbol{g}}_{s+j}\in {\boldsymbol{g}\choose s+j}}{s+j\choose s}
|{Y}_{\bar{\boldsymbol{g}}_{s+j}}(i)|
\end{align*}
because each $\bar{\boldsymbol{g}}_{s+j}\in  {\boldsymbol{g}\choose s+j}$ on the right side
is counted ${s+j\choose s}$ times to be $\bar{\boldsymbol{g}}_s\cup \boldsymbol{t}_j$ on the left side.
We further have 
\begin{align}
|{Q}_{\boldsymbol{g}}(i)|
&=\sum_{s=2}^{r}\sum_{j=0}^{r-s}\sum_{\bar{\boldsymbol{g}}_{s+j}\in {\boldsymbol{g}\choose s+j}}(-1)^j{s+j\choose s}
|{Y}_{\bar{\boldsymbol{g}}_{s+j}}(i)|. 
\end{align} 
Hence, $|{Q}_{\boldsymbol{g}}(i)|$ can be regarded as the sum of all elements in the $(r-1)\times (r-1)$
upper triangular matrix below
$$\small{
\left(\begin{array}{ccc}
\sum\limits_{\bar{\boldsymbol{g}}_{2}\in {\boldsymbol{g}\choose 2}}(-1)^0{2\choose 2}|{Y}_{\bar{\boldsymbol{g}}_{2}}(i)|&\cdots&\sum\limits_{\bar{\boldsymbol{g}}_{r}\in {\boldsymbol{g}\choose r}}(-1)^{r-2}{r\choose 2}|{Y}_{\bar{\boldsymbol{g}}_{r}}(i)|\\
\vdots&\ddots
&\vdots\\
\sum\limits_{\bar{\boldsymbol{g}}_{r}\in {\boldsymbol{g}\choose r}}(-1)^{0}{r\choose r}|{Y}_{\bar{\boldsymbol{g}}_{r}}(i)|&\cdots&0
\end{array}\right)},
$$
where the row corresponds to the index $s$ and the column corresponds to the index $j$ in (4.12), respectively.
Summing these elements again according to all back diagonal lines, 
\begin{align*}
|{Q}_{\boldsymbol{g}}(i)|&=\sum_{s=2}^{r}\sum_{\bar{\boldsymbol{g}}_{s}\in {\boldsymbol{g}\choose s}}
\sum_{j=2}^{s}(-1)^{s-j}{s\choose j}|{Y}_{\bar{\boldsymbol{g}}_{s}}(i)|
=\sum_{s=2}^{r}\sum_{\bar{\boldsymbol{g}}_{s}\in {\boldsymbol{g}\choose s}}
(-1)^s(s-1)|{Y}_{\bar{\boldsymbol{g}}_{s}}(i)|,
\end{align*}
where $\sum_{j=2}^{s}(-1)^{s-j}{s\choose j}=(-1)^s(s-1)$.
According to the expressions of $|{Y}_{\bar{\boldsymbol{g}}_{s}}(i)|$ in~(3.16) for 
$2\leqslant s\leqslant r$, the term $|{Y}_{\bar{\boldsymbol{g}}_{2}}(i)|$ dominates the sum on 
the right side of~(4.12). As $s=2$, as the equations shown in~(3.20) and~(3.23), 
we complete the proof of Claim~4.1.2, where the last term $O(n^{r-3})$
comes from those terms when $s\neq 2$. 
\end{proof}

\subsubsection{$\mathbb{E}[\Delta {Y}_{\boldsymbol{f}_m}(i)\,|\mathcal{F}_i]$ 
and $\mathbb{E}[\Delta {W}_{\boldsymbol{f},L,k}(i)\,|\mathcal{F}_i]$ again
}

On the basis of the equations 
in~(4.3) and~(4.10), the enumerative formulas in Claim~4.1.1 and Claim~4.1.2,
in order to obtain $\mathbb{E}[\Delta {Y}_{\boldsymbol{f}_m}(i)\,|\mathcal{F}_i]$
for any $\boldsymbol{f}_m\in {K}_m(i)$ 
with $2\leqslant m\leqslant r-1$,
and $\mathbb{E}[\Delta {W}_{\boldsymbol{f},L,k}(i)\,|\mathcal{F}_i]$ for any 
 $\boldsymbol{f}\in {Q}(i)$, 
$L\in \mathcal{L}$
and  $0\leqslant k\leqslant e_L-2$,
we also require the following estimates hold on
$|\Upsilon_{\boldsymbol{f}}(i)|$ in~(4.2), 
 $|\Psi_{L'}(i)|$ in~(4.6) and $|\Lambda_{\boldsymbol{f},L,k-1}(i)|$ in~(4.9).

\vskip 0.3cm
\noindent\textbf{Claim 4.1.3} \  
Let $\mathcal{A}(i)$ denote the event that
\begin{align}
|\Upsilon_{\boldsymbol{f}}(i)|&=O\bigl(n^{\alpha'+r-3}(nt_{{M}})^{\ell-2}\bigr),\\
|\Psi_{L'}(i)|&=O\bigl(n^{\alpha'+r-3}(nt_{{M}})^{\ell-2}\bigr),\\
\mathbbm{1}_{\{k\geqslant 1\}}\cdot|\Lambda_{\boldsymbol{f},L,k-1}(i)|&=O\bigl(n^{\alpha'+(r-1)(e_L-k)-2}(nt_{{M}})^{k-1}\bigr),
\end{align}
hold for any $\boldsymbol{f}\in {Q}(i)$, $L\in \mathcal{L}$
and $L'\in {W}_{\boldsymbol{f},L,k}(i)$ with $0\leqslant k\leqslant e_L-2$, 
where $\alpha_0= \frac{\ell-2}{\ell-1}$ and $\alpha'\in (\alpha_0,\alpha)$.
On the condition of  the event $\mathcal{G}(i)$, 
the event $\mathcal{A}(i)$ holds with probability $1-o(n^{-\mu})$, 
that is,
\begin{align*}
\mathbb{P}\bigl[ \bar{\mathcal{A}}(i)\,|\mathcal{G}(i)\bigr]=o(n^{-\mu}).
\end{align*}
The proof of each equation in~(4.13)-(4.15) is left to be analyzed in Section 5.2.

Assuming the event $\mathcal{G}(i)\cap \mathcal{A}(i)$ holds, we now have 
$\mathbb{E}[\Delta {Y}_{\boldsymbol{f}_m}(i)\,|\mathcal{F}_i]$
and $\mathbb{E}[\Delta {W}_{\boldsymbol{f},L,k}(i)\,|\mathcal{F}_i]$ as follows.

\vskip 0.3cm
\noindent\textbf{Claim 4.1.4}\ On the condition of the event $\mathcal{G}(i)\cap \mathcal{A}(i)$,
for any $\boldsymbol{f}_m\in {K}_m(i)$ 
with $2\leqslant m\leqslant r-1$,
\begin{align*}
\mathbb{E}[\Delta {Y}_{\boldsymbol{f}_m}(i)\,|\mathcal{F}_i]
&= \frac{{y}_m'}{n^2}+ {O\Bigl(\sigma_i^\ell n^{\alpha+r-m-3}(nt_{{M}})^{\ell-2}\Bigr)}.
\end{align*}

\begin{proof}[Proof of Claim~4.1.4]
Note that
${\epsilon}_{y_2}=O(\sigma_i n^{\alpha+r-3})$ in~(3.23), 
\begin{align*}
\sum_{L\in \mathcal{L}}\epsilon_{w_{L,e_L-2}}=O\Bigl(\sigma_i^\ell n^{\alpha+r-3}(nt_{{M}})^{\ell-2}\Bigr)
\end{align*} in~(3.24), and
the equation of $|\Upsilon_{\boldsymbol{f}}(i)|$ for any $\boldsymbol{f}\in {Q}(i)$ shown in~(4.13),
\begin{align*}
\sum_{L\in \mathcal{L}}\epsilon_{w_{L,e_L-2}}\gg {\epsilon}_{y_2}\ \text{and}\ 
\sum_{L\in \mathcal{L}}\epsilon_{w_{L,e_L-2}}\gg |\Upsilon_{\boldsymbol{f}}(i)|.
\end{align*}

As the equation of $\mathbb{E}[\Delta {Y}_{\boldsymbol{f}_m}(i)\,|\mathcal{F}_i]$ 
in~(4.3), $|{Q}_{\boldsymbol{f}_m^{c}}(i)|$ in Claim~4.1.1,
and $\hat{\epsilon}_{q}=\hat{\epsilon}_{y_m}$ for any $2\leqslant m\leqslant r-1$ in~(3.28),
$\mathbb{E}[\Delta {Y}_{\boldsymbol{f}_m}(i)\,|\mathcal{F}_i]$ is rearranged as 
\begin{align*}
&\mathbb{E}[\Delta {Y}_{\boldsymbol{f}_m}(i)\,|\mathcal{F}_i]\notag\\
&=-  \frac{(1+O(\hat{\epsilon}_{y_m})){y}_m}{{q}}\biggl[
\biggl({r\choose 2}-
{m\choose 2}\biggr){y}_2+\sum_{L\in \mathcal{L}}{w}_{L,e_L-2}+O\Bigl(\sigma_i^\ell n^{\alpha+r-3}(nt_{{M}})^{\ell-2}\Bigr)\biggr].
\end{align*} 
Since  $\tilde{\xi}_i=O((nt_M)^{\ell-2})$ in~(3.12), ${q}$ in~(3.19), ${y}_m$
in~(3.20), ${w}_{L,k}$
in~(3.21) and $\hat{\epsilon}_{y_m}$ in~(3.28), we have 
 \begin{align}
\sum_{L\in \mathcal{L}} \frac{{w}_{L,e_L-2}}{{q}}&= \frac{\tilde{\xi}_i}{n^2},\notag\\
\frac{\hat{\epsilon}_{y_m} {y}_m}{{q}}\Bigl( {y}_2+ \sum_{L\in \mathcal{L}} {w}_{L,e_L-2}\Bigr)
&=O\Bigl(\sigma_i n^{\alpha+r-m-3}(nt_{{M}})^{\ell-2}\Bigr),\notag
\end{align}
and then 
\begin{align*}
\mathbb{E}[\Delta {Y}_{\boldsymbol{f}_m}(i)\,|\mathcal{F}_i]
&=-\biggl[\biggl({r\choose 2}-
{m\choose 2}\biggr) \frac{{y}_m{y}_2}{{q}}+ \frac{{y}_m\tilde{\xi}_i}{n^2}\biggr]+ 
O\Bigl(\sigma_i^\ell n^{\alpha+r-m-3}(nt_{{M}})^{\ell-2}\Bigr).
\end{align*}
Here, it is checked that 
\begin{align*}
\frac{{y}_m'}{n^2}=-\biggl[\biggl({r\choose 2}-
{m\choose 2}\biggr) \frac{{y}_m{y}_2}{{q}}+ \frac{{y}_m\tilde{\xi}_i}{n^2}\biggr],
\end{align*}
by differentiating $y_m$  in~(3.20) with respect to $t$,
the equations of $p_i$, $\xi_i$ and $\tilde{\xi}_i$ in~(3.5),~(3.9) and~(3.12),
to complete the proof of Claim~4.1.4.
\end{proof}

\vskip 0.3cm
\noindent\textbf{Claim 4.1.5}\ On the condition of the event 
$\mathcal{G}(i)\cap \mathcal{A}(i)$, 
\begin{align*}
\mathbb{E}[\Delta {W}_{\boldsymbol{f},L,k}(i)\,|\mathcal{F}_i]
&=\begin{cases} \frac{{w}_{L,0}'}{n^2}+O\bigl(\sigma_i^\ell n^{\alpha+(r-1)e_L-r-3}(nt_{{M}})^{\ell-2}\bigr),\quad \text{if }k=0;\\
 \frac{{w}_{L,k}'}{n^2}\pm \beta_k\sigma_i^{k+1} n^{\alpha+(r-1)(e_L-k)-r-2}(nt_{{M}})^{k-1},\quad \text{if\ }k\geqslant 1.
\end{cases}
\end{align*}

\begin{proof}[Proof of Claim~4.1.5]
In the case of $k\geqslant 1$, we consider $\mathbb{E}[\Delta {W}_{\boldsymbol{f},L,k}(i)|\mathcal{F}_i]$ 
by~(4.10). 
By the equations shown in~(3.23) and~(3.24),
\begin{align*}
\epsilon_{y_{2}}&=\sigma_i n^{\alpha+r-3},\\
\epsilon_{w_{L,k-1}}&=O\bigl(\sigma_i^{k+1} n^{\alpha+(r-1)(e_L-k)-2}(nt_{{M}})^{k-1}\bigr),\\
\sum_{L\in \mathcal{L}}\epsilon_{w_{L,e_L-2}}&=O\bigl(\sigma_i^\ell n^{\alpha+r-3}(nt_{{M}})^{\ell-2}\bigr).
\end{align*}
By the equations of $|\Upsilon_{\boldsymbol{g}}(i)|$, $|\Psi_{L'}(i)|$ 
and $|\Lambda_{\boldsymbol{f},L,k-1}(i)|$ in~(4.13)-(4.15), 
we further have 
\begin{align*}
\sum_{L\in \mathcal{L}}\epsilon_{w_{L,e_L-2}}\gg \epsilon_{y_{2}},&\qquad \sum_{L\in \mathcal{L}}\epsilon_{w_{L,e_L-2}}\gg |\Upsilon_{\boldsymbol{g}}(i)|,\\ 
\sum_{L\in \mathcal{L}}\epsilon_{w_{L,e_L-2}}\gg |\Psi_{L'}(i)|&,\qquad \epsilon_{w_{L,k-1}}\gg |\Lambda_{\boldsymbol{f},L,k-1}(i)|.
\end{align*}  
Hence, as the equation of $\mathbb{E}[\Delta {W}_{\boldsymbol{f},L,k}(i)\,|\mathcal{F}_i]$ in~(4.10),
 $|{Q}_{\boldsymbol{g}}(i)|$ in Claim~4.1.2, $\mathbb{E}[\Delta {W}_{\boldsymbol{f},L,k}(i)|\mathcal{F}_i]$ 
 is rearranged as 
\begin{align}
&\mathbb{E}[\Delta {W}_{\boldsymbol{f},L,k}(i)\,|\mathcal{F}_i]\notag\\
&=- \frac{({w}_{L,k}+O(\epsilon_{w_{L,k}}))}{(1+O(\hat{\epsilon}_{q})){q}}
\biggl[\bigl(e_L-k-1\bigr)\biggl(\binom{r}{2}{y}_2+\sum_{L\in \mathcal{L}}{w}_{L,e_L-2}\biggr)
+O\Bigl(\sigma_i^\ell n^{\alpha+r-3}(nt_{{M}})^{\ell-2}\Bigr)\biggr]\notag\\
&\qquad+\frac{1}{(1+O(\hat{\epsilon}_{q})){q}}
\biggl[\bigl(e_L-k\bigr){w}_{L,k-1}\pm
2(e_L-k)\epsilon_{w_{L,k-1}}\biggr].
\end{align}

By the equations of $\tilde{\xi}_i$, ${q}$, ${y}_2$ and ${w}_{L,k}$
  shown in~(3.12), (3.19)-(3.21), we have
\begin{align*}
 \frac{\binom{r}{2}{y}_2+\sum_{L\in \mathcal{L}}{w}_{L,e_L-2}}{{q}}&= 
 \frac{r^2(r-1)^2}{2n^2p_i}+ \frac{\tilde{\xi}_i}{n^2}.
\end{align*}
The first term on the right side of the equation in~(4.16) is reduced to
\begin{align}
&- \bigl(w_{L,k}+O(\epsilon_{w_{L,k}})\bigr)
\biggl[\bigl(e_L-k-1\bigr)\biggl( \frac{r^2(r-1)^2}{2n^2p_i}+ \frac{\tilde{\xi}_i}{n^2}\biggr)\bigl(1-O(\hat{\epsilon}_{q})\bigr)
+ O\Bigl(\sigma_i^\ell n^{\alpha-3}(nt_{{M}})^{\ell-2}\Bigr)\biggr]\notag\\
=&- w_{L,k}\bigl(e_L-k-1\bigr)\biggl( \frac{r^2(r-1)^2}{2n^2p_i}+ \frac{\tilde{\xi}_i}{n^2}\biggr)
+O\Bigl(\sigma_i^\ell n^{\alpha+(r-1)(e_L-k)-r-3}(nt_{{M}})^{k+\ell-2}\Bigr),
\end{align}
where the term $\hat{\epsilon}_{q}\cdot ( \frac{r^2(r-1)^2}{2n^2p_i}+ \frac{\tilde{\xi}_i}{n^2})$
is absorbed into $O(\sigma_i^\ell n^{\alpha-3}(nt_{{M}})^{\ell-2})$ by $\tilde{\xi}_i=O((nt_M)^{\ell-2})$ in~(3.12)
and $\hat{\epsilon}_{q}$ in~(3.29),
and then the term $\epsilon_{w_{L,k}}\cdot ( \frac{r^2(r-1)^2}{2n^2p_i}+ \frac{\tilde{\xi}_i}{n^2})$
is absorbed into 
\begin{align*}
O\bigl(\sigma_i^\ell n^{\alpha+(r-1)(e_L-k)-r-3}(nt_{{M}})^{k+\ell-2}\bigr)
\end{align*} by $\tilde{\xi}_i$ in~(3.12) and $\epsilon_{w_{L,k}}$ in~(3.24).
Specially, 
\begin{align*}
\epsilon_{w_{L,\ell-2}}\cdot \Bigl( \frac{r^2(r-1)^2}{2n^2p_i}+ \frac{\tilde{\xi}_i}{n^2}\Bigr)=
\Theta\bigl(\sigma_i^\ell n^{\alpha+(r-1)(e_L-k)-r-3}(nt_{{M}})^{k+\ell-2}\bigr)
\end{align*}
based on $\epsilon_{w_{L,k}}$ in~(3.24).

By the equations of ${w}_{L,k-1}$, $\epsilon_{w_{L,k-1}}$ and $\beta_k$
in~(3.21),~(3.24) and~(3.27), we have
\begin{align*}
{w}_{L,k-1}\cdot \hat{\epsilon}_{q}\ll \epsilon_{w_{L,k-1}}\ \text{and}\ 
\frac{2r!(e_L-k)\beta_{k-1}}{p_i^{\binom{r}{2}}\xi_i}<\beta_k.
\end{align*}
Hence, 
\begin{align*}
\frac{(e_L-k)\cdot {w}_{L,k-1}\cdot \hat{\epsilon}_{q}}{q}+\frac{2(e_L-k)\cdot\epsilon_{w_{L,k-1}}}{q}<\beta_{k}\sigma_i^{k+1} n^{\alpha+(r-1)(e_L-k)-r-2}(nt_{\textsc{M}})^{k-1}.
\end{align*}It follows that
the second term on the right side of the equation in~(4.16) is reduced to
\begin{align}
\bigl(e_L-k\bigr) \frac{{w}_{L,k-1}}{{q}}\pm
\beta_{k}\sigma_i^{k+1} n^{\alpha+(r-1)(e_L-k)-r-2}(nt_{\textsc{M}})^{k-1}.
\end{align}
Here, it is also checked that
\begin{align*}
 \frac{{w}_{L,k}'}{n^2}&=-{w}_{L,k}\bigl(e_L-k-1\bigr)\biggl( \frac{r^2(r-1)^2}{2n^2p_i}+ \frac{\tilde{\xi}_i}{n^2}\biggr)
 +\bigl(e_L-k\bigr) \frac{{w}_{L,k-1}}{{q}}
\end{align*}
by differentiating ${w}_{L,k}$ in (3.21) with respect to $t$. Combining the equations shown in (4.16)-(4.18), 
we complete the proof of Claim~4.2.4 in the case of $k\geqslant 1$,
where the error term in~(4.17) is negligible compared to the one
in~(4.18) because
\begin{align*}
\beta_k\sigma_i^{k+1} n^{\alpha+(r-1)(e_L-k)-r-2}(nt_{{M}})^{k-1}
\gg \sigma_i^\ell n^{\alpha+(r-1)(e_L-k)-r-3}(nt_{{M}})^{k+\ell-2}
\end{align*}
is reduced to $\beta_k\sigma_i^{k+1}\gg \sigma_i^\ell(\log n)^{-\lambda(\ell-1)}$
by $t_{{M}}$ in~(3.26)
that is clearly true
as $\sigma_i=\Theta(\log n)$ in~(3.25) and $\lambda> \frac{\ell}{\ell-1}$ in~(3.27).

For $k=0$, we have  the only error term in~(4.17), which is $O(\sigma_i^\ell n^{\alpha+(r-1)e_L-r-3}(nt_{{M}})^{\ell-2})$.
\end{proof}

\subsubsection{$|\Delta {Y}_{\boldsymbol{f}_{m}}(i)|$ 
and $|\Delta {W}_{\boldsymbol{f},L,k}(i)|$}

Next consider $|\Delta {Y}_{\boldsymbol{f}_{m}}(i)|$ for any $\boldsymbol{f}_m\in {K}_m(i)$ 
with $2\leqslant m\leqslant r-1$, and  $|\Delta {W}_{\boldsymbol{f},L,k}(i)|$
for any $\boldsymbol{f}\in {Q}(i)$, 
$L\in \mathcal{L}$ and $0\leqslant k\leqslant e_L-2$.
Let $\boldsymbol{e}_{i+1}\in {Q}(i)$ denote the available $r$-set that is added to $\mathbb{H}(i)$.
In fact, based on the analysis in Section~4.1.1 and Section~4.1.2,
we have the ways  to bound the one-step changes of $\Delta {Y}_{\boldsymbol{f}_{m}}(i)$ 
and  $\Delta {W}_{\boldsymbol{f},L,k}(i)$.
We formally characterize these ways into three boundedness parameters in higher
uniformities below.

Define  $\Pi_{\boldsymbol{f}_{m},\boldsymbol{e}_{i+1}}(i)$ 
as the collection of $(r-m)$-sets 
in ${Y}_{\boldsymbol{f}_{m}}(i)$ 
that are removed due to the addition of the available edge $\boldsymbol{e}_{i+1}$
into $\mathbb{H}(i)$, where $|\boldsymbol{e}_{i+1}\cap \boldsymbol{f}_{m}|\leqslant 1$
because $\boldsymbol{f}_{m}\in K_m(i+1)$.
Thus, we have
\begin{align}
|\Delta {Y}_{\boldsymbol{f}_{m}}(i)|&= O\bigl(|\Pi_{\boldsymbol{f}_{m},\boldsymbol{e}_{i+1}}(i)|\bigr).
\end{align}

To bound $\Delta {W}_{\boldsymbol{f},L,k}(i)$ for any $\boldsymbol{f}\in {Q}(i)$, 
$L\in \mathcal{L}$ and $0\leqslant k\leqslant e_L-2$,
 we firstly define  $\Pi_{\boldsymbol{f},L,k,\boldsymbol{e}_{i+1}}(i)$ 
as the collection of $L'\in {W}_{\boldsymbol{f},L,k}(i)$
with  $\boldsymbol{e}_{i+1}\in(L'\cap {Q}(i))\backslash \{\boldsymbol{f}\}$,
that is, $|\Pi_{\boldsymbol{f},L,k,\boldsymbol{e}_{i+1}}(i)|$
is the number of $L'$ that is removed from ${W}_{\boldsymbol{f},L,k}(i)$
due to $\boldsymbol{e}_{i+1}$ equals one available $r$-set in $L'\backslash \{\boldsymbol{f}\}$.
Specially, in
the case of $k\geqslant 1$, 
 $|\Pi_{\boldsymbol{f},L,k-1,\boldsymbol{e}_{i+1}}(i)|$ 
 refers to the growth in $\Delta {W}_{\boldsymbol{f},L,k}(i)$.
We also define 
$\Phi_{\boldsymbol{f},L,k,\boldsymbol{e}_{i+1}}(i)$ 
 as the collection of $L'\in {W}_{\boldsymbol{f},L,k}(i)$
that is removed from ${W}_{\boldsymbol{f},L,k}(i)$ 
due to the addition of $\boldsymbol{e}_{i+1}\notin L'$ to $\mathbb{H}(i)$.
Thus, we have
\begin{align}
|\Delta {W}_{\boldsymbol{f},L,k}(i)|&= 
O\Bigl(|\Pi_{\boldsymbol{f},L,k,\boldsymbol{e}_{i+1}}(i)|+|\Phi_{\boldsymbol{f},L,k,\boldsymbol{e}_{i+1}}(i)|+\mathbbm{1}_{\{k\geqslant 1\}}\cdot|\Pi_{\boldsymbol{f},L,k-1,\boldsymbol{e}_{i+1}}(i)|\Bigr).
\end{align}

On the basis of the equations in~(4.19) and~(4.20), 
in order to bound $|\Delta {Y}_{\boldsymbol{f}_{m}}(i)|$ 
 and $|\Delta {W}_{\boldsymbol{f},L,k}(i)|$,
we also require the following estimates to hold on
$|\Pi_{\boldsymbol{f}_{m},\boldsymbol{e}_{i+1}}(i)|$, $|\Phi_{\boldsymbol{f},L,k,\boldsymbol{e}_{i+1}}(i)|$
and $|\Pi_{\boldsymbol{f},L,k,\boldsymbol{e}_{i+1}}(i)|$. 

\vskip 0.3cm
\noindent \textbf{Claim 4.1.6}\ 
Let $\mathcal{B}(i)$ denote
the event that 
\begin{align}
\mathbbm{1}_{\{|\boldsymbol{f}_{m}\cap\boldsymbol{g}|\leqslant 1\}}\cdot|\Pi_{\boldsymbol{f}_{m},\boldsymbol{g}}(i)|&= O\bigl( n^{ \alpha'+r-m-1}\bigr),\\
\mathbbm{1}_{\{\boldsymbol{f}\neq\boldsymbol{g}\}}\cdot|\Pi_{\boldsymbol{f},L,k,\boldsymbol{g}}(i)|&= O\bigl( n^{\alpha'+(r-1)(e_L-k-2-\mathbbm{1}_{\{k<e_L-2\}})}\bigr),\\
\mathbbm{1}_{\{\boldsymbol{f}\neq\boldsymbol{g}\}}\cdot|\Phi_{\boldsymbol{f},L,k,\boldsymbol{g}}(i)|&= O\bigl( n^{\alpha'+ (r-1)(e_L-k)-r-1}(nt_{{M}})^{k}\bigr),
\end{align}
hold for any $\boldsymbol{f}_{m}\in {K}_m(i)$ with $2\leqslant m\leqslant r-1$, $\boldsymbol{f}$,
$\boldsymbol{g}\in {Q}(i)$, 
$L\in \mathcal{L}$ and $0\leqslant k\leqslant e_L-2$, where $\alpha'\in (\alpha_0,\alpha)$.
On the condition of the event $\mathcal{G}(i)$, the event $\mathcal{B}(i)$
holds with probability $1-o(n^{-\mu})$, that is, 
\begin{align*}
\mathbb{P}[\bar{\mathcal{B}}_i\,|\, \mathcal{G}_i]=o(n^{-\mu}).
\end{align*}
The proof of each equation in (4.21)-(4.23) is left to be analyzed in Section 5.2.

Assuming the event $\mathcal{G}(i)  \cap\mathcal{B}(i)$ holds, by Claim~4.1.6 and
the equations shown in~(4.19) and~(4.20), we now have 
$|\Delta {Y}_{\boldsymbol{f}_{m}}(i)|$ and $|\Delta {W}_{\boldsymbol{f},L,k}(i)|$
as follows.

\vskip 0.3cm
\noindent\textbf{Claim 4.1.7}\ On the condition of the event 
 $\mathcal{G}(i)  \cap\mathcal{B}(i)$, 
\begin{align*}
|\Delta {Y}_{\boldsymbol{f}_{m}}(i)|&=O\bigl(n^{ \alpha'+r-m-1}\bigr),\\
|\Delta {W}_{\boldsymbol{f},L,k}(i)|&=O\bigl(n^{ \alpha'+(r-1)(e_L-k)-r-1}(nt_{{M}})^k\bigr).
\end{align*}

\subsection{The one-step (expected) changes of $ {Y}^{\pm}_{{\boldsymbol{f}_{m}}}(i)$
and ${W}^{\pm}_{\boldsymbol{f},L,k}(i)$
}

\subsubsection{$\mathbb{E}\bigl[|\Delta {Y}^{\pm}_{{\boldsymbol{f}_{m}}}(i)||\mathcal{F}_i\bigr]$ 
and $|\Delta {Y}^{\pm}_{{\boldsymbol{f}_{m}}}(i)|$}

By~(2.1) and~(2.2), we have ${Y}^{\pm}_{{\boldsymbol{f}_{m}}}(i)=
\pm[|{Y}_{{\boldsymbol{f}_{m}}}(i)|-{y}_m(t)]-\epsilon_{y_m}(t)$
and 
\begin{align}
\Delta {Y}^{\pm}_{{\boldsymbol{f}_{m}}}(i)
&=\pm \biggl[\Delta {Y}_{{\boldsymbol{f}_{m}}}(i)- \frac{{y}_m'}{n^2}\biggr]- \frac{\epsilon_{y_m}'}{n^2}
+O\biggl(\sup_{\varsigma\in[t, t+ \frac{1}{n^2}]} \frac{|{y_m}''(\varsigma)|+|\epsilon_{y_m}''(\varsigma)|}{n^4}\biggr).
\end{align}

\vskip 0.3cm

\noindent\textbf{Claim 4.2.1}\  
For any $t\in [0, t_{{M}}]$ and $\boldsymbol{f}_{m}\in K_m(i)$ with $2\leqslant m\leqslant r-1$,
we have $\{{Y}^{\pm}_{{\boldsymbol{f}_{m}}}(i)\}_{i\geqslant 0}$
is a supermartingale
and 
\begin{align*}\mathbb{E}\bigl[|\Delta {Y}^{\pm}_{{\boldsymbol{f}_{m}}}(i)|\,|\mathcal{F}_i\bigr]=
O\Bigl(n^{\alpha+r-m-1}\bigl(n^\alpha+n^2t\bigr)^{-1}\Bigr).
\end{align*}
\begin{proof}[Proof of Claim~4.2.1]
In order to establish the desired supermartingale condition 
\begin{align*}
\mathbb{E}[\Delta {Y}^{\pm}_{{\boldsymbol{f}_{m}}}(i)\,|\mathcal{F}_i]\leqslant 0
\end{align*}
by Claim~4.1.4 and the equation shown in~(4.24), it is sufficient to show that the error function
$\epsilon_{y_m}$ in~(3.23) satisfies the following inequality 
\begin{align}
\epsilon_{y_m}'\gg \max\biggl\{ \sigma_i^\ell n^{\alpha+r-m-1}(nt_{{M}})^{\ell-2}, \ \sup_{\zeta\in [t,t+ \frac{1}{n^2}]} \frac{|{y}_m''(\zeta)|+|\epsilon_{y_m}''(\zeta)|}{n^2}\biggr\}.
\end{align}

 As the equations of 
 $\epsilon_{y_m}$ and $\sigma_i$ shown in~(3.23) and~(3.25), we have
\begin{align*}
\epsilon_{y_m}'= \frac{ n^{\alpha+r-m+1}}{n^\alpha+n^2t}\ 
\text{and}\ \epsilon_{y_m}''= - \frac{ n^{\alpha+r-m+3}}{(n^\alpha+n^2t)^2}.
\end{align*} 
It is clearly correct that $\epsilon_{y_m}'\gg  \sup_{\zeta\in [t,t+ \frac{1}{n^2}]} n^{-2}\cdot|\epsilon_{y_m}''(\zeta)|$.
We further check 
\begin{align*}
\epsilon_{y_m}'\gg  \sigma_i^\ell n^{\alpha+r-m-1}(nt_{{M}})^{\ell-2}
\end{align*}
in~(4.25) with $t_{{M}}$ in~(3.26), which is true if
$
1\gg \sigma_i^\ell\cdot (\log n)^{-\lambda (\ell-1)}
$
that holds as $\sigma_i=\Theta(\log n)$ in~(3.25) and $\lambda> \frac{\ell}{\ell-1}$ in~(3.27).

As the equation of ${y}_m$ shown in~(3.20), differentiate ${y}_m$ with respect to $t$, 
and it is routine to check that $|{y}_m'|=O(n^{r-m}(1+(nt)^{\ell-2}))=O(n^{r-m+\ell-2}t_M^{\ell-2})$
and $|{y}_m''|=O(n^{ r-m+\ell-2}t_M^{\ell-3})$ such that
\begin{align*}
\epsilon_{y_m}'\gg  \sup_{\zeta\in [t,t+ \frac{1}{n^2}]} n^{-2}\cdot|{y}_m''(\zeta)|.
\end{align*}

Finally, it is natural to obtain
$\mathbb{E}[|\Delta {Y}^{\pm}_{{\boldsymbol{f}_{m}}}(i)|\, |\mathcal{F}_i]=
O\bigl( \frac{|\epsilon_{y_m}'|}{n^2}\bigr)$. 
\end{proof}

\vskip 0.3cm
\noindent\textbf{Claim 4.2.2}\ Assume the event 
 $\mathcal{G}(i)  \cap\mathcal{B}(i)$ holds. Then,
\begin{align*}
|\Delta {Y}_{\boldsymbol{f}_{m}}^{\pm}|=O\bigl(n^{ \alpha'+r-m-1}\bigr)
\end{align*}
 for any $\boldsymbol{f}_{m}\in K_m(i)$ with $2\leqslant m\leqslant r-1$.
 
 \begin{proof}[Proof of Claim 4.2.2]\ 
 Following the equation in~(4.24) and the proof of Claim~4.2.1,
we  have shown
\begin{align*}
|\Delta {Y}_{\boldsymbol{f}_{m}}^{\pm}(i)|&\leqslant |\Delta {Y}_{\boldsymbol{f}_{m}}(i)| +O\Bigl(\max\Bigl\{ \frac{|y_m'|}{n^2}, \frac{|\epsilon_{y_{m}}'|}{n^2}\Bigr\}\Bigr).
\end{align*} 
As the equations of $y_{m}$ in~(3.20) and $\epsilon_{y_{m}}$ in~(3.23),
we have obtained 
\begin{align*}
 \frac{|y_{m}'|}{n^2}=O\bigl(n^{r-m-2}(1+(nt)^{\ell-2})\bigr)\quad\text{and}\quad 
 \frac{|\epsilon_{y_{m}}'|}{n^2}=O\bigl(n^{\alpha+r-m-1}(n^\alpha+n^2t)^{-1}\bigr),
 \end{align*}
 which are negligible compared with $|\Delta {Y}_{\boldsymbol{f}_{m}}|=O(n^{\alpha'+r-m-1})$
 shown in Claim~4.1.7.
 \end{proof}

\subsubsection{$\mathbb{E}\bigl[|\Delta {W}^{\pm}_{\boldsymbol{f},L,k}(i)|| \mathcal{F}_i\bigr]$
and $|\Delta {W}^{\pm}_{\boldsymbol{f},L,k}(i)|$}

By~(2.1) and~(2.2), we have ${W}^{\pm}_{\boldsymbol{f},L,k}(i)=
\pm[|{W}_{\boldsymbol{f},L,k}(i)|-{w}_{L,k}(t)]-\epsilon_{w_{L,k}}(t)$
and 
\begin{align}
\Delta {W}^{\pm}_{\boldsymbol{f},L,k}(i)
&=\pm \biggl[\Delta {W}_{\boldsymbol{f},L,k}(i)- \frac{{w}_{L,k}'}{n^2}\biggr]- \frac{\epsilon_{w_{L,k}}'}{n^2}
+O\biggl(\sup_{\varsigma\in[t, t+ \frac{1}{n^2}]} \frac{|{w}_{L,k}''(\varsigma)|+|\epsilon_{w_{L,k}}''(\varsigma)|}{n^4}\biggr).
\end{align}

\vskip 0.3cm
\noindent\textbf{Claim 4.2.3}\   
For any $t\in [0, t_{{M}}]$,  $\boldsymbol{f}\in {Q}(i)$, $L\in \mathcal{L}$ and $0\leqslant k\leqslant e_L-2$, 
we have $\{{W}^{\pm}_{\boldsymbol{f},L,k}(i)\}_{i\geqslant 0}$ is a supermartingale
and 
\begin{align*}
\mathbb{E}\bigl[|\Delta {W}_{\boldsymbol{f},L,k}^{\pm}(i)|\, |\mathcal{F}_i\bigr]&
=O\Bigl(\sigma_i^{k+1}  n^{\alpha+(r-1)(e_L-k)-r-1}(n^\alpha+n^2t)^{-1}(nt_{{M}})^k\Bigr).
\end{align*}

\begin{proof}[Proof of Claim~4.2.3] 
In order to establish the desired supermartingale condition 
\begin{align*}
\mathbb{E}\bigl[\Delta {W}_{\boldsymbol{f},L,k}^{\pm}\,| \mathcal{F}_i\bigr]\leqslant 0
\end{align*}
by Claim~4.1.5 and the equation shown in~(4.26),   we will show that the error function 
 $\epsilon_{w_{L,k}}$ in~(3.24) satisfies the following inequalities 
\begin{align}
\epsilon_{w_{L,k}}'&\begin{cases} \gg\sigma_i^\ell n^{\alpha+(r-1)e_L-r-1}(nt_{{M}})^{\ell-2},\quad \text{if }k=0; \\>\beta_k\sigma_i^{k+1} n^{\alpha+(r-1)(e_L-k)-r}(nt_{{M}})^{k-1},\quad \text{if }k\geqslant 1.\end{cases}
\end{align}
and   
 \begin{align}
\epsilon_{w_{L,k}}'&\gg 
\sup_{\zeta\in [t,t+ \frac{1}{n^2}]} \frac{|{w}_{L,k}''(\zeta)|+|\epsilon_{w_{L,k}}''(\zeta)|}{n^2},
\end{align}
Differentiating $\epsilon_{w_{L,k}}$ in~(3.24) with respect to $t$, we have
\begin{align*}
\epsilon_{w_{L,k}}'&=(k+2)\beta_k\sigma_i^{k+1} n^{\alpha+(r-1)(e_L-k)-r+1}(n^\alpha+n^2t)^{-1}(nt_{{M}})^k.
\end{align*}
Now we check that the equations shown in~(4.27) and~(4.28) are true.

Firstly, we check 
\begin{align*}
\epsilon_{w_{L,0}}'\gg 
\sigma_i^\ell n^{\alpha+(r-1)e_L-r-1}(nt_{{M}})^{\ell-2}
\end{align*} in~(4.27)
for $k=0$,  which is true if
$2\beta_0\sigma_i\gg \sigma_i^\ell (\log n)^{-\lambda(\ell-1)}$
that holds as $\sigma_i=\Theta(\log n)$ in~(3.25), $t_M$ in~(3.26)
and $\lambda> \frac{\ell}{\ell-1}$ in~(3.27).
For $k\geqslant 1$, the inequality
in~(4.27) is reduced to $(k+2)(n^2t_{{M}})>(n^\alpha+n^2t)$,
which is clearly true as $\alpha\in (0,1)$ and $t\in [0,t_M]$ with $t_M$ in~(3.26).

Taking the derivative of $\epsilon_{w_{L,k}}'$ with respect to $t$ again, we have 
\begin{align*}
\epsilon_{w_{L,k}}''&=O\bigl( \sigma_i^{k+1} n^{\alpha+(r-1)(e_L-k)-r+3}(n^\alpha+n^2t)^{-2}(nt_{\textsc{M}})^k\bigr),
\end{align*}
which implies that 
$\epsilon_{w_{L,k}}'\gg 
\sup_{\zeta\in [t,t+ \frac{1}{n^2}]} n^{-2}\cdot{|\epsilon_{w_{L,k}}''(\zeta)|}$ in~(4.28).
Note that we can not change the power of $\sigma_i$ in the expression of
$\epsilon_{w_{L,k}}$ defined in~(3.24).

At last, we check  
$\epsilon_{w_{L,k}}'\gg 
\sup_{\zeta\in [t,t+ \frac{1}{n^2}]} n^{-2}\cdot{|{w}_{L,k}''(\zeta)|}$.
Differentiating ${w}_{L,k}$ in~(3.21) with respect to $t$, 
as the equations of $p_i$, $\xi_i$ and $\tilde{\xi}_i$ in~(3.5),~(3.9) and~(3.12),
it is checked that
\begin{align*}
{w}_{L,k}'&=\mathbbm{1}_{\{k\geqslant 1\}}\cdot O\bigl(n^{(r-1)(e_L-k)+k-r}t^{k-1}\bigr)+
O\bigl(n^{(r-1)(e_L-k)+k-r}(1+(nt)^{\ell-2})t^{k}\bigr).
\end{align*}
and
\begin{align*}
{w}_{L,k}''&=O\bigl(n^{(r-1)(e_L-k)+\ell+k-r-2}t^{\ell+k-3}\bigr)\\
&\quad+\mathbbm{1}_{\{k\geqslant 1\}}\cdot O\bigl(n^{(r-1)(e_L-k)+k-r}\bigl(1+(nt)^{\ell-2}\bigr)t^{k-1}\bigr)\\
&\quad+\mathbbm{1}_{\{k\geqslant 2\}}\cdot O\bigl(n^{(r-1)(e_L-k)+k-r}t^{k-2}\bigr).
\end{align*} 
Note that all these terms $n^{-2}\cdot {w}_{L,k}''$ for $t\in [0,t_M]$ are negligible compared to $\epsilon_{w_{L,k}}'$ to complete the 
proof the inequality in~(4.28).

Furthermore, it is  natural to obtain 
$\mathbb{E}[|\Delta {W}_{\boldsymbol{f},L,k}^{\pm}(i)|\, |\mathcal{F}_i]=O\bigl( \frac{|\epsilon_{w_{L,k}}'|}{ n^2}\bigr)$.
\end{proof}

 \noindent\textbf{Claim 4.2.4}\ Assume the event 
 $\mathcal{G}(i)\cap\mathcal{B}(i)$ holds. Then,
 \begin{align*}
 |\Delta {W}_{\boldsymbol{f},L,k}^{\pm}(i)|=
 O\bigl(n^{ \alpha'+(r-1)(e_L-k)-r-1}(nt_{{M}})^k\bigr)
 \end{align*}
 for any $\boldsymbol{f}\in Q(i)$, $L\in \mathcal{L}$ and $0\leqslant k\leqslant e_L-2$.

 \begin{proof}[Proof of Claim~4.2.4]\ 
 Following the equation shown in~(4.26) and the proof of Claim~4.2.3, we have shown
 \begin{align}
|\Delta {W}_{\boldsymbol{f},L,k}^{\pm}|&\leqslant |\Delta {W}_{\boldsymbol{f},L,k}(i)| +O\Bigl(\max\Bigl\{\frac{|w_{L,k}'|}{n^2}, \frac{|\epsilon_{w_{L,k}}'|}{n^2}\Bigr\}\Bigr).
\end{align}

As the equations of $w_{L,k}$ shown in~(3.21) and $\epsilon_{w_{L,k}}$ shown in~(3.24),
we have obtained that 
\begin{align*}
 \frac{|{w}_{L,k}'|}{n^2}=\mathbbm{1}_{\{k\geqslant 1\}}\cdot O\bigl(n^{(r-1)(e_L-k)+k-r-2}t^{k-1}\bigr)+
O\bigl(n^{(r-1)(e_L-k)+k-r-2}(1+(nt)^{\ell-2})t^{k}\bigr)
\end{align*}
 and 
\begin{align*}
 \frac{|\epsilon_{w_{L,k}}'|}{n^2}&=O\bigl(\sigma_i^{k+1}  n^{\alpha+(r-1)(e_L-k)-r-1}(n^{\alpha}+n^2t)^{-1}(nt_{{M}})^k\bigr),
\end{align*} 
which are all negligible compared with $|\Delta {W}_{\boldsymbol{f},L,k}(i)|
=O(n^{\alpha'+(r-1)(e_L-k)-r-1}(nt_{{M}})^k)$ shown in Claim~4.1.7.
 \end{proof}

\subsection{Proof of Theorem~3.4}

For $i\geqslant 0$,
define the event  $\mathcal{G}^{+}(i)=\mathcal{G}(i) \cap \mathcal{A}(i) \cap \mathcal{B}(i)$. 
Recall that $|{Y}_{\boldsymbol{f}_m}(0)|= \binom{n-m}{r-m}$ for any $\boldsymbol{f}_m\in {K}_m(0)$
with $2\leqslant m\leqslant r-1$, and $|{W}_{\boldsymbol{f},L,k}(0)|=\mathbbm{1}_{\{k=0\}}\cdot\mathcal|{N}_{\boldsymbol{f},L}|$, 
where $|\mathcal{N}_{\boldsymbol{f},L}|$ is shown in~(3.7) for any $\boldsymbol{f}\in {Q}(0)$, 
$L\in \mathcal{L}$ and $0\leqslant k\leqslant e_L-2$.
Firstly, we show the event $\mathcal{G}(0)$ always holds. 
By the expressions of ${y}_m$ and ${w}_{L,k}$ 
in~(3.20) and~(3.21), and the expressions of $\epsilon_{y_m}$ and $\epsilon_{w_{L,k}}$
in~(3.23) and~(3.24) when $t=0$, we have
\begin{align}
{Y}_{\boldsymbol{f}_m}^{\pm}(0)&= \pm[|{Y}_{\boldsymbol{f}_m}(0)|-{y}_m(0)]-\epsilon_{y_{m}}(0)\notag\\
&< - \frac 12 \sigma_0 n^{\alpha+r-m-1},\\
{W}_{\boldsymbol{f},L,k}^{\pm}(0)&=\pm[|{W}_{\boldsymbol{f},L,k}(0)|-{w}_{L,k}(0)]-\epsilon_{w_{L,k}}(0)\notag\\
&= \mathbbm{1}_{\{k=0\}}\cdot O\bigl(n^{(r-1)e_L-r-1}\bigr)-\epsilon_{w_{L,k}}(0)\notag\\
&<- \frac 12 \beta_k\sigma_0^{k+2}n^{\alpha+(r-1)(e_L-k)-r-1}(nt_{{M}})^k, 
\end{align}
which implies that $|{Y}_{{\boldsymbol{f}}_m}(0)|={y}_m(0)\pm \epsilon_{{y}_m}(0)$ and 
$|{W}_{{\boldsymbol{f}},L,k}(0)|={w}_{L,k}(0)\pm \epsilon_{w_{L,k}}(0)$ hold. 
The equation of $|{Q}(0)|={q}(0)\pm\epsilon_{q}(0)$ shown in (3.15) follows from 
the analysis of the equation shown in~(3.28).

Assume 
the event $\mathcal{G}(i-1)$ holds for  $i\geqslant 1$.
By Claim~4.1.3, Claim~4.1.6 and the law of total probability, 
 $\mathcal{G}^+(i-1)$ holds with probability $1-o(n^{-\mu})$.
Hence, 
\begin{align}
\mathbb{P}[\bar{ \mathcal{G}}({{M}})]&\leqslant 
\mathbb{P}\bigl[ \bar{\mathcal{G}}({i})\cap \mathcal{G}^+({i-1})
\text{ for some }1\leqslant i\leqslant {M}\bigr]+o\bigl(n^{-\mu}\bigr),
\end{align}
where the second term $o(n^{-\mu})$ refers to 
the probability that the event $\mathcal{G}^+({i-1})$
does not hold.

Let the stopping time $\tau_{\boldsymbol{f}_m}$ be the minimum of ${M}$, the last step $i\geqslant 1$
such that $\boldsymbol{f}_m\in {K}_m(i)$, and the maximum $i\geqslant 1$ that
the event $\bar{\mathcal{G}}^{+}({i-1})$ holds. Similarly, let the stopping
time $\tau_{\boldsymbol{f}}$ be the minimum of ${M}$, the last step $i\geqslant 1$ 
such that $\boldsymbol{f}\in {Q}(i)$, and the maximum $i\geqslant 1$ that the event $ \bar{\mathcal{G}}^{+}({i-1})$ holds. 
It follows that
\begin{align}
&\mathbb{P}\Bigl[ \bar{\mathcal{G}}({i})\cap \mathcal{G}^+({i-1})\text{ for some }1\leqslant i\leqslant {M}\Bigr]\notag\\
&\leqslant \sum_{2\leqslant m\leqslant r-1}\sum_{\boldsymbol{f}_m\in {K}_m(i)}\sum_{\flat\in\{+,-\}}\mathbb{P}\Bigl[{Y}_{\boldsymbol{f}_m}^{\flat}
\bigl(i\wedge \tau_{\boldsymbol{f}_m}\bigr)\geqslant 0\text{ for some }i\geqslant 1\Bigr]\notag\\
&\qquad+\sum_{\boldsymbol{f}\in {Q}(i)}\sum_{L\in \mathcal{L}}\sum_{0\leqslant k\leqslant e_L-2}
\sum_{\flat\in\{+,-\}}\mathbb{P}\Bigl[{W}_{\boldsymbol{f},L,k}^{\flat}\bigl(i\wedge \tau_{\boldsymbol{f}}\bigr)\geqslant 0\text{ for some }i\geqslant 1\Bigr].
\end{align}

In order to bound the above probability,  the following tail probability
formula in~\cite{freed75,glock20,boh19,warnke16} is applied to  each inequality in~(4.33).
The following tail probability is due to Freedman~\cite{freed75},
which is called to be Freedman's martingale concentration inequality
and was originally stated for martingales, while the proof for supermartingales is the same.
It is also applied to study
the existence of sparse Steiner triple systems 
to approximately answer a conjecture of Erd\H{o}s in the previous works of Glock
et al.~\cite{glock20} and of Bohman and Warnke~\cite{boh19}.
The simplified version of the Freedman's martingale concentration inequality
here~\cite{boh19,glock20,warnke16} improves the Azuma-Hoeffding inequality when the expected one-step change
is much smaller than the worst case one, that is, 
$\mathbb{E}[|\Delta {S}(i)|\, |\mathcal{F}_i]\ll \max_{i\geqslant 0}|\Delta {S}(i)|$.


\vskip 0.3cm

\noindent\textbf{Lemma 4.3.1} 
Let $\{{S}(i)\}_{i\geqslant 0}$ be a supermartingale with respect to the filtration
$\mathcal{F}=(\mathcal{F}_i)_{i\geqslant 0}$. Let $\Delta {S}(i)={S}(i+1)-{S}(i)$.
Suppose that $\sum_{i\geqslant 0}\mathbb{E}[|\Delta {S}(i)|\, |\mathcal{F}_i]\leqslant C$ and 
$\max_{i\geqslant 0}|\Delta {S}(i)|\leqslant V$.
Then, for any $z>0$,
\begin{align*}
\mathbb{P}\Bigl[{S}(i)\geqslant {S}(0)+z\text{ for some }i\geqslant 1\Bigr]
\leqslant \exp\biggl[- \frac{z^2}{2V(C+z)}\biggr].
\end{align*}

As the equation shown in~(4.30), we have 
\begin{align*}
{Y}_{\boldsymbol{f}_m}^{\flat}(0)
< - \frac 12 \sigma_0 n^{\alpha+r-m-1}
\end{align*} for $\flat\in\{+,-\}$.
By Claim~4.2.1 and Claim~4.2.2, it follows that 
the sequence $\{{Y}_{\boldsymbol{f}_m}^{\flat}(i\wedge \tau_{\boldsymbol{f}_m})\}_{i\geqslant 0}$
is a supermartingale,  
\begin{align*}
\sum_{i\geqslant 0}\mathbb{E}[|\Delta {Y}_{\boldsymbol{f}_m}^{\flat}(i)|\, |\mathcal{F}_i]
&=n^2t\cdot O\bigl(n^{\alpha+r-m-1}(n^\alpha+n^2t)^{-1}\bigr)=O\bigl(n^{\alpha+r-m-1}\bigr)
\end{align*}
and $|\Delta {Y}_{\boldsymbol{f}_{m}}^{\flat}(i)|=O(n^{ \alpha'+r-m-1})$ 
with $\alpha'\in (\alpha_0,\alpha)$.
Applying Lemma~4.3.1 with
\begin{align*}
z=\frac 12 \sigma_0 n^{\alpha+r-m-1},\ 
C=O(n^{\alpha+r-m-1})\text{ and }
V=O(n^{\alpha'+r-m-1})
\end{align*}
 to  ${Y}_{\boldsymbol{f}_{m}}^{\flat}(i)$,
we have
\begin{align}
&\mathbb{P}\bigl[{Y}_{\boldsymbol{f}_m}^{\flat}(i\wedge \tau_{\boldsymbol{f}_m})\geqslant 0\text{ for some }i\geqslant 1\bigr]\notag\\
&\leqslant\mathbb{P}\bigl[{Y}_{\boldsymbol{f}_m}^{\flat}(i\wedge \tau_{\boldsymbol{f}_m})\geqslant {Y}_{\boldsymbol{f}_m}^{\flat}(0)+z\bigr]\notag\\
&\leqslant
\exp\biggl[- \Omega\biggl( \frac{\log^2 n\cdot n^{2\alpha+2r-2m-2}}{n^{\alpha'+r-m-1}\cdot n^{\alpha+r-m-1}}\biggr)\biggr]\notag\\
&= n^{-\omega(1)}.
\end{align}

Similarly, we have 
\begin{align*}
{W}_{\boldsymbol{f},L,k}^{\flat}(0)< 
- \frac 12\beta_k\sigma_0^{k+2}n^{\alpha+(r-1)(e_L-k)-r-1}(nt_{{M}})^k
\end{align*}
by~(4.31)  for $\flat\in\{+,-\}$.
By Claim~4.2.3 and Claim~4.2.4, the sequence $\{{W}_{\boldsymbol{f},L,k}^{\flat}(i\wedge \tau_{\boldsymbol{f}})\}_{i\geqslant 0}$
is also a supermartingale,
\begin{align*}
\sum_{i\geqslant 0}\mathbb{E}[|\Delta {W}_{\boldsymbol{f},L,k}^{\flat}(i)|\, |\mathcal{F}_i]
&=n^2t\cdot O\bigl(\sigma_i^{k+1}  n^{\alpha+(r-1)(e_L-k)-r-1}(n^\alpha+n^2t)^{-1}(nt_{{M}})^k\bigr)\\
&=O\bigl(\sigma_i^{k+1}n^{\alpha+(r-1)(e_L-k)-r-1}(nt_{{M}})^k\bigr)
\end{align*} 
 and $|\Delta {W}_{\boldsymbol{f},L,k}^{\pm}(i)|=
 O(n^{ \alpha'+(r-1)(e_L-k)-r-1}(nt_{{M}})^k)$
 with $\alpha'\in(\alpha_0,\alpha)$.
Applying Lemma~4.3.1 with
\begin{align*}
z&= \frac 12\beta_k\sigma_0^{k+2}n^{\alpha+(r-1)(e_L-k)-r-1}(nt_{{M}})^k,\\
C&=O\bigl(\sigma_i^{k+1}n^{\alpha+(r-1)(e_L-k)-r-1}(nt_{{M}})^k\bigr),\\
V&=O\bigl(n^{\alpha'+(r-1)(e_L-k)-r-1}(nt_{{M}})^k\bigr) 
\end{align*}
to ${W}_{\boldsymbol{f},L,k}^{\flat}(i)$, it follows that
\begin{align}
&\mathbb{P}\Bigl[{W}_{\boldsymbol{f},L,k}^{\flat}(i\wedge \tau_{\textbf{f}})\geqslant 0\text{ for some }i\geqslant 0\Bigr]\notag\\
&\leqslant\mathbb{P}\bigl[{W}_{\boldsymbol{f},L,k}^{\flat}(i\wedge \tau_{\textbf{f}})\geqslant
 {W}_{\boldsymbol{f},L,k}^{\flat}(0)+z\bigr]\notag\\
&\leqslant\exp\biggl[- \Omega\biggl(\frac{ (\log n)^{k+2}\cdot n^{2\alpha+2(r-1)(e_L-k)-2r-2}(nt_{{M}})^{2k}}{n^{\alpha'+(r-1)(e_L-k)-r-1}(nt_{{M}})^k\cdot n^{\alpha+(r-1)(e_L-k)-r-1}(nt_{{M}})^k}\biggr)\biggr]\notag\\
&= n^{-\omega(1)}.
\end{align}

Taking a union bound of the probabilities shown in~(4.34) for any 
$\boldsymbol{f}_m\in K_m(i)$ with $2\leqslant m\leqslant r-1$, and the ones in~(4.35)
for any $\boldsymbol{f}\in {Q}(i)$, $L\in \mathcal{L}$ and $0\leqslant k\leqslant e_L-2$, 
combining the equation shown in~(4.33), $|\mathcal{L}|=O(1)$, $|K_m(i)|< n^m$
and $|{Q}(i)|<n^r$, we have $\mathbb{P}\bigl[ \bar{\mathcal{G}}({i})\cap \mathcal{G}^+({i-1})
\text{ for some }1\leqslant i\leqslant {M}\bigr]=o(n^{-\mu})$.
Finally, as the equations shown in~(4.32), we complete the proof of 
Theorem~3.4 by $\mathbb{P}[ \bar{\mathcal{G}}({{M}})]=o(n^{-\mu})$.

\section{Claim~4.1.3 and Claim~4.1.6}

In this section, 
assume that the event $\mathcal{G}(i)$ holds 
for an integer $i$  with $0\leqslant i< {M}$, 
and we prove Claim~4.1.3 and Claim~4.1.6. 

\subsection{ A special  injective function}

Given an $r$-graph $H$ with 
$\max \{v_H,e_H\}\leqslant 2r\ell$, let $\hat{H}\subsetneq H$ be a subhypergraph, and 
$\theta$ be a fixed  injection from $V(\hat{H})$ to the vertex-set $[n]$ of $\mathbb{H}(i)$.

An injection $\psi$  maps an $r$-set $\boldsymbol{f}=\{x_1,\cdots,x_r\}$
as $\psi(\boldsymbol{f})=\psi(x_1)\ldots\psi(x_r)$ and maps a collection $U$ of $r$-sets as 
 $\psi(U)=\cup_{\boldsymbol{f}\in U}\psi(\boldsymbol{f})$.
 Define a set of injections $\psi$ extended from $\theta$, denoted as 
$\Psi_{\theta,\hat{H},H}(i)$, to be
\begin{align}
\Psi_{\theta,\hat{H},H}(i)=\Bigl\{\psi:\, V(H)\rightarrow [n]\text{ with }
\psi|_{V(\hat{H})}\equiv\theta\text{ and }\psi(H\backslash \hat{H})\subseteq \mathbb{H}(i)\Bigr\},
\end{align}which implies that the value $|\Psi_{\theta,\hat{H},H}(i)|$ counts copies of $H$
in $\mathbb{H}(i)\cup \theta(\hat{H})$ and these copies all have the fixed image $\theta(\hat{H})$
of $\hat{H}$ under the fixed injection $\theta$ as a subhypergraph.

 \begin{lemma}
With the above notation, let $\mathcal{C}(i)$ denote the event that, for any $\epsilon>0$, an
$r$-graph pair $(\hat{H}, H)$ and  an
injection $\theta: V(\hat{H})\rightarrow [n]$, the set $\Psi_{\theta,\hat{H},H}(i)$ 
defined in~(5.1) satisfies
\begin{align*}
\bigl|\Psi_{\theta,\hat{H},H}(i)\bigr|\leqslant n^{\epsilon}\cdot \max_{\hat{H}\subseteq \widetilde{H}\subseteq H}
n^{[v_H-(r-1)e_{H}]+[(r-1)e_{\widetilde{H}}-v_{\widetilde{H}}]}\cdot(nt)^{e_H-e_{\widetilde{H}}}.
\end{align*}  
On the condition of  the event $\mathcal{G}(i)$, 
the event $\mathcal{C}(i)$ holds with probability $1-o(n^{-\mu})$, 
that is,
\begin{align*}
\mathbb{P}\bigl[ \bar{\mathcal{C}}(i)\,|\mathcal{G}(i)\bigr]=o(n^{-\mu}).
\end{align*}
\end{lemma}

\begin{proof}[Proof of Lemma~5.1] For any $i$ with $0\leqslant i< {M}$, suppose 
that the event $\mathcal{G}({i})$  holds.
For any collection ${U}\subseteq \binom{[n]}{r}$ of $r$-sets, 
we have 
\begin{align*}
\mathbb{P}\bigl[{U}\subseteq \mathbb{H}(i)|\mathcal{G}(i)\bigr]
=O\Bigl( i^{|{U}|}\cdot \Bigl( \frac{1}{|{Q}(i)|}\Bigr)^{|{U}|}\Bigr)
\end{align*}
 because there are 
at most $i^{|{U}|}$ ways that all $r$-sets in ${U}$ appear in $\mathbb{H}(i)$,
and each $r$-set is added to $\mathbb{H}(i)$ at some step $j\leqslant i$ with probability at most $1/|{Q}(i)|$.
Since $p_i$ in~(3.5) and $\xi_i$ in~(3.9) are $\Theta(1)$ when $t\in [0,t_{{M}}]$ as shown in
Remark~3.5,
we have $|{Q}(i)|=\Theta(n^r)$ as the equation shown in~(3.19),
and it follows that
\begin{align}
\mathbb{P}\bigl[{U}\subseteq \mathbb{H}(i)\,|\mathcal{G}(i)\bigr]=O\Bigl(\bigl(n^{2-r}t\bigr)^{|{U}|}\Bigr).
\end{align}

Hence, on the definition of  $\Psi_{\theta,\hat{H},H}(i)$ in~(5.1),
for any fixed
$r$-graph pair $(\hat{H}, H)$,  an
injection $\theta: V(\hat{H})\rightarrow [n]$, 
and $s\geqslant 1$ an integer that will be defined later,
we have
\begin{align}
\mathbb{E}\Bigl[\Bigl|\bigl(\Psi_{\theta,\hat{H},H}(i)\bigr)^s\,\Bigr|\Big|\mathcal{G}(i)\Bigr]
&=\sum_{(\psi_1,\cdots,\psi_s)\in (\Psi_{\theta,\hat{H},H}(i))^s}\mathbb{E}\Bigl[
\prod_{j\in [s]}\Bigl(\mathbbm{1}_{\{\psi_j(H\backslash \hat{H})\subseteq \mathbb{H}(i)\}}\Bigr)\,\Big|\mathcal{G}(i)\Bigr],
\end{align}
where $(\Psi_{\theta,\hat{H},H}(i))^s$ is the $s$-th power set of $\Psi_{\theta,\hat{H},H}(i)$,
the sum is over all $s$-tuples $(\psi_1,\cdots,\psi_s)$ with the injection 
$\psi_j: V({H})\rightarrow [n]$ satisfying $\psi_j|_{V(\hat{H})}\equiv \theta$
 for $j\in [s]$.

 Let ${E}_j=\psi_j(H\backslash \hat{H})$ with $j\in [s]$. As the equation shown in~(5.2),
\begin{align*}
\mathbb{E}\Bigl[\mathbbm{1}_{\{\psi_j(H\backslash \hat{H})\subseteq \mathbb{H}(i)\}}\,\Big|\mathcal{G}(i)\Bigr]=O\Bigl(\bigl(n^{2-r}t\bigr)^{|{E}_j|}\Bigr).
\end{align*}\
 We rearrange the equation in~(5.3) to 
 \begin{align}
\mathbb{E}\Bigl[\Bigl|\bigl(\Psi_{\theta,\hat{H},H}(i)\bigr)^s\,\Bigr|\Big|\mathcal{G}(i)\Bigr]
&= O\Bigl(\sum_{(\psi_1,\cdots,\psi_s)\in (\Psi_{\theta,\hat{H},H}(i))^s}\bigl(n^{2-r}t\bigr)^{\bigl|\cup_{j\in [s]}|{E}_j|\bigr|}\Bigr).
\end{align}

Conversely, for any $s$-tuple of injections $(\psi_1,\cdots,\psi_s)\in (\Psi_{\theta,\hat{H},H}(i))^s$,
there must exist 
$\widetilde{H}_j$  with
$\hat{H}\subseteq \widetilde{H}_j\subseteq H$ satisfying
\begin{align}
{E}_j\cap \bigl(\cup_{k\in [j-1]}{E}_k\bigr)\cong \widetilde{H}_j\backslash \hat{H}\quad\text{and}\quad 
{E}_j\backslash \cup_{k\in [j-1]}{E}_k\cong H\backslash \widetilde{H}_j,
\end{align}
where $j\in [s]$.  By double-counting arguments,
for any $s$-tuples of graphs $(\widetilde{H}_1,\ldots,\widetilde{H}_s)$ 
with $\widetilde{H}_j$ satisfying $\hat{H}\subseteq \widetilde{H}_j\subseteq H$, 
we have at most $\prod_{j\in [s]}n^{v_H-v_{\widetilde{H}_j}}$ ways to generate 
one $(\psi_1,\cdots,\psi_s)\in (\Psi_{\theta,\hat{H},H}(i))^s$. 

Therefore, for any $s$-tuples of graphs $(\widetilde{H}_1,\ldots,\widetilde{H}_s)$ 
with $\widetilde{H}_j$ satisfying $\hat{H}\subseteq \widetilde{H}_j\subseteq H$,
we rewrite the equation in~(5.4) to 
\begin{align*}
\mathbb{E}\Bigl[\Bigl|\bigl(\Psi_{\theta,\hat{H},H}(i)\bigr)^s\Bigr|\Big|\mathcal{G}(i)\Bigr]&
=O\Bigl( \sum_{(\widetilde{H}_1,\ldots,\widetilde{H}_s)}\prod_{j\in [s]}n^{v_H-v_{\widetilde{H}_j}}\bigl(n^{2-r}t\bigr)^{e_H-e_{\widetilde{H}_j}}\Bigr)\\
&= O\biggl(\Bigl(\max_{\hat{H}\subseteq \widetilde{H}\subseteq H}n^{v_H-v_{\widetilde{H}}}\bigl(n^{2-r}t\bigr)^{e_H-e_{\widetilde{H}}}\Bigr)^s\biggr)\\
&= O\biggl(\Bigl(\max_{\hat{H}\subseteq \widetilde{H}\subseteq H}n^{[v_H-(r-1)e_H]+[(r-1)e_{\widetilde{H}}-v_{\widetilde{H}}]}\bigl(nt\bigr)^{e_H-e_{\widetilde{H}}}\Bigr)^s\biggr),
\end{align*}
where the second equality is correct because the ways to choose 
$(\widetilde{H}_1,\ldots,\widetilde{H}_s)$ is finite for the fixed
$r$-graph pair $(\hat{H}, H)$.
For any $\epsilon>0$, choosing  $s=\bigl\lceil(2+2r\ell+\mu)/\epsilon\bigr\rceil$, by Markov's inequality
of conditional probability, we have
 \begin{align*}
 &\mathbb{P}\biggl[\Bigl|\Psi_{\theta,\hat{H},H}(i)\Bigr|> 
 n^\epsilon\cdot\max_{\hat{H}\subseteq \widetilde{H}\subseteq H}n^{[v_H-(r-1)e_H]+[(r-1)e_{\widetilde{H}}-v_{\widetilde{H}}]}\cdot\bigl(nt\bigr)^{e_H-e_{\widetilde{H}}}\,\Big|\mathcal{G}(i)\biggr]\\
 &\leqslant\mathbb{P}\biggl[\Bigl|\bigl(\Psi_{\theta,\hat{H},H}(i)\bigr)^s\Bigr|> 
 n^{\epsilon s}\cdot\Bigl(\max_{\hat{H}\subseteq \widetilde{H}\subseteq H}n^{[v_H-(r-1)e_H]+[(r-1)e_{\widetilde{H}}-v_{\widetilde{H}}]}\cdot\bigl(nt\bigr)^{e_H-e_{\widetilde{H}}}\Bigr)^s\,\Big|\mathcal{G}(i)\biggr]\\
& =O\bigl( n^{-\epsilon s}\bigr)\\
&=O\bigl( n^{-(2+2r\ell+\mu)}\bigr).
 \end{align*}
Finally, by taking a  union bound  for all possible $(i, \hat{H}, H, \theta)$ of
the above inequality,
which has at most $o(n^{2+2r\ell})$ ways because there are $n^2t_M=o(n^{1+ \frac{1}{\ell-1}})$
ways to choose $i$ with $t_M$ in~(3.26), 
 and $O(n^{2r\ell})$ ways to choose $(\hat{H}, H, \theta)$
as $\max\{v_H,e_H\}\leqslant 2r\ell$, we have 
\begin{align*}
\mathbb{P}[\bar{\mathcal{C}}(i)\,|\mathcal{G}(i)]=o(n^{-\mu})
\end{align*}
to complete the proof of Lemma~5.1. 
\end{proof}

\begin{remark}
We make no efforts to 
characterize the optimal upper bound of $\max \{v_H,e_H\}$ in Section~5.1. 
Indeed, following the same discussions,
$\max \{v_H,e_H\}$ could be 
any  constant, while the value $2r\ell$ is enough
to satisfy the number of vertices in the structures counted by the overcounting and boundedness parameters 
in Claim~4.1.3 and Claim~4.1.6. 
\end{remark}

\subsection{View the parameters in a different way}
By Lemma~5.1, for $1\leqslant i<M$, assuming the event $\mathcal{G}(i)\cap\mathcal{C}(i)$ holds,
we introduce some notation to recharacterize the overcounting and boundedness parameters  in Claim~4.1.3 and Claim~4.1.6.

\begin{definition}
For any given $L\in \mathcal{L}$, $0\leqslant k\leqslant e_L-2$, 
a collection of nonempty available $r$-sets ${U}\subseteq {Q}(i)$
and $\boldsymbol{f}_m=\{x_1,\cdots,x_m\}\subseteq \boldsymbol{f}\in {Q}(i)\backslash {U}$ 
such that 
\begin{align*}
|{U}|+\mathbbm{1}_{\{m\geqslant 1\}}\leqslant e_L-k.
\end{align*}
In particular, $\boldsymbol{f}_m=\boldsymbol{f}=\emptyset$ when $m=0$.
Let $\Gamma_{{U},\boldsymbol{f}_m,L,k}(i)$ be defined as 
\begin{align}
\Gamma_{{U},\boldsymbol{f}_m,L,k}(i)=\Bigl\{L'\in \mathcal{F}_L:\, 
{U}\subseteq L'\cap {Q}(i),\ \boldsymbol{f}_m\subseteq V(L')
\text{ and }\bigl|L'\cap \mathbb{H}(i)\bigr|\geqslant k\Bigr\}.
\end{align}
\end{definition}

\begin{lemma}
Assume that  the event $\mathcal{G}(i)\cap \mathcal{C}(i)$ holds. 
With the above notation, it follows that 
\begin{align*}
\bigl|\Gamma_{{U},\boldsymbol{f}_m,L,k}(i)\bigr|=O\biggl(n^{\alpha'+(r-1)\cdot
\bigl(e_L-k-|{U}|-\mathbbm{1}_{\{|{U}|+\mathbbm{1}_{\{m\geqslant 1\}}< e_L-k\}}\bigr)
-\mathbbm{1}_{\{m\geqslant 1\}}\cdot m}\biggr),
\end{align*}where $\alpha'\in (\alpha_0,\alpha)$ and $\alpha_0= \frac{\ell-2}{\ell-1}$. 
\end{lemma}

The proof of Lemma~5.4 is left to Section~5.3. Firstly, we use it 
to prove the equations shown in~(4.21) and~(4.22) in Claim~4.1.6.

\begin{proof}[Proof of (4.21) in Claim~4.1.6.]\ For any $\boldsymbol{f}_{m}\in {K}_m(i)$ with $2\leqslant m\leqslant r-1$
and $\boldsymbol{g}\in {Q}(i)$, if $|\boldsymbol{f}_{m}\cap\boldsymbol{g}|\leqslant 1$,
by the set $\Pi_{\boldsymbol{f}_m,\boldsymbol{g}}(i)$ defined in Section~4.1.4 
and  $\Gamma_{{U},\boldsymbol{f}_m,L,k}(i)$ defined in~(5.6), we have
\begin{align*}
|\Pi_{\boldsymbol{f}_m,\boldsymbol{g}}(i)|= \sum_{L\in \mathcal{L}}|\Gamma_{\{\boldsymbol{g}\},\boldsymbol{f}_m,L,e_L-2}(i)|.
\end{align*}
For any $L\in \mathcal{L}$, substituting ${U}=\{\boldsymbol{g}\}$, $2\leqslant m\leqslant r-1$
and $k=e_L-2$ in Lemma~5.4,
we have 
\begin{align*}
\mathbbm{1}_{\{|{U}|+\mathbbm{1}_{\{m\geqslant 1\}}< e_L-k\}}=0,
\end{align*}
and then $|\Gamma_{\{\boldsymbol{g}\},\boldsymbol{f}_m,L,e_L-2}(i)|=O(n^{\alpha'+r-m-1})$.
Since $|\mathcal{L}|=O(1)$, we have the equation in~(4.21) is true.
\end{proof}

\begin{proof}[Proof of (4.22) in Claim~4.1.6]\ For any  $\boldsymbol{f}$, $\boldsymbol{g}\in {Q}(i)$, 
$L\in \mathcal{L}$ and $0\leqslant k\leqslant e_L-2$, if $\boldsymbol{f}\neq\boldsymbol{g}$,
by the set $\Pi_{\boldsymbol{f},L,k,\boldsymbol{g}}(i)$ defined in Subsection~4.1.4
and $\Gamma_{{U},\boldsymbol{f}_m,L,k}(i)$ defined in~(5.6),
we have 
\begin{align*}
|\Pi_{\boldsymbol{f},L,k,\boldsymbol{g}}(i)|=|\Gamma_{\{\boldsymbol{f},\boldsymbol{g}\},\emptyset,L,k}(i)|.
\end{align*}
Substituting ${U}=\{\boldsymbol{f},\boldsymbol{g}\}$ and $m=0$ in Lemma~5.4,
it follows that
\begin{align*}
\mathbbm{1}_{\{|{U}|+\mathbbm{1}_{\{m\geqslant 1\}}< e_L-k\}}=\mathbbm{1}_{\{k< e_L-2\}},
\end{align*}
and then
\begin{align*}
|\Gamma_{\{\boldsymbol{f},\boldsymbol{g}\},\emptyset,L,k}(i)|=O\bigl(n^{\alpha'+(r-1)\cdot(e_L-k-2-\mathbbm{1}_{\{k< e_L-2\}})}\bigr)
\end{align*}
to show the equation~(4.22) is true.
\end{proof}

With the help of $\Gamma_{{U},\boldsymbol{f}_m,L,k}(i)$ defined in~(5.6), 
we further define some sets  to collect overlapping copies of
the pair $L$ and $K$ that are extended from certain $r$-sets in $Q(i)$
for any $(L,K)\in \mathcal{L}^2$, 
where $\mathcal{L}^2=\{(L,K)\,| L,K\in \mathcal{L}\}$ is the power set of $\mathcal{L}$.

\begin{definition}
Given $r$-sets $\boldsymbol{f},\boldsymbol{g}$ and $\boldsymbol{h}\in {Q}(i)$
such that  $\boldsymbol{f}\neq \boldsymbol{h}$ and 
$\boldsymbol{g}\neq \boldsymbol{h}$, $(L,K)\in \mathcal{L}^2$, define the set
${R}_{\{\boldsymbol{f},\boldsymbol{h}\},L,k,\{\boldsymbol{g},\boldsymbol{h}\},K}(i)$
to be
\begin{align}
&{R}_{\{\boldsymbol{f},\boldsymbol{h}\},L,k,\{\boldsymbol{g},\boldsymbol{h}\},K}(i)\notag\\
&\hskip 20pt=\Bigl\{(L',K'):\ (L',K')\in  \Gamma_{\{\boldsymbol{f},\boldsymbol{h}\},\emptyset,L,k}\times {\Gamma}_{\{\boldsymbol{g},\boldsymbol{h}\},\emptyset,K,e_K-2}\text{ and }   L'\neq K'\Bigr\}.
\end{align}
\end{definition}

By adding constraint on  choosing $\boldsymbol{h}$, define
${R}^{+}_{\boldsymbol{f},L,k,\boldsymbol{g},K}(i)$ as
\begin{align}
{R}^{+}_{\boldsymbol{f},L,k,\boldsymbol{g},K}(i)&=\Bigl\{(L',K'):\  (L',K')\in  \bigcup_{\{\boldsymbol{h}\}=L'\cap K'\cap {Q}(i)\backslash \{\boldsymbol{f},\boldsymbol{g}\}}{R}_{\{\boldsymbol{f},\boldsymbol{h}\},L,k,\{\boldsymbol{g},\boldsymbol{h}\},K}(i)\Bigr\}. 
\end{align}
By further adding constraint on choosing $\boldsymbol{g}$, define ${R}^{++}_{\boldsymbol{f},L,k,K}(i)$
as
\begin{align}
{R}^{++}_{\boldsymbol{f},L,k,K}&=\Bigl\{(L',K'):\  (L',K')\in \bigcup_{\boldsymbol{g}\in L'\cap {Q}(i)}{R}^{+}_{\boldsymbol{f},L,k,\boldsymbol{g},K}(i)\Bigr\}.
\end{align}

\begin{remark}As the equations shown in~(5.7)-(5.9), note that $e_L, e_K\leqslant \ell$ 
for any $(L,K)\in \mathcal{L}^2$ and $|\mathcal{L}|=O(1)$,
then it follows that
\begin{align*}
\bigl|{R}^{+}_{\boldsymbol{f},L,k,\boldsymbol{g},K}(i)\bigr|\leqslant\ell^2\cdot \bigl|{R}_{\{\boldsymbol{f},\boldsymbol{h}\},L,k,\{\boldsymbol{g},\boldsymbol{h}\},K}(i)\bigr|
\end{align*}
and
\begin{align*}
\bigl|{R}^{++}_{\boldsymbol{f},L,k,K}(i)\bigr|\leqslant \ell\cdot \bigl|{R}^{+}_{\boldsymbol{f},L,k,\boldsymbol{g},K}(i)\bigr|.
\end{align*}
\end{remark}

\begin{lemma}
Assume that  the event $\mathcal{G}(i)\cap \mathcal{C}(i)$ holds.
With the above notation, it follows that 
\begin{align*}
|{R}_{\{\boldsymbol{f},\boldsymbol{h}\},L,k,\{\boldsymbol{g},\boldsymbol{h}\},K}(i)|&=O\Bigl( n^{ \alpha' +(r-1)(e_L-k)-r-1}(nt_{{M}})^k\Bigr),
\end{align*}
where $\alpha'\in (\alpha_0,\alpha)$.
\end{lemma}

The proof of Lemma~5.7 is left to Section~5.3. 
We firstly use it 
to prove the equations~(4.13)-(4.15) in Claim~4.1.3 and the equation~(4.23) in Claim~4.1.6.

\begin{proof}[Proof of Claim~4.1.3.]\ 
For any $\boldsymbol{f}\in {Q}(i)$, $L\in \mathcal{L}$
and $L'\in {W}_{\boldsymbol{f},L,k}(i)$ with $0\leqslant k\leqslant e_L-2$, by the definitions of
 $\Upsilon_{\boldsymbol{f}}(i)$ in~(4.2), $\Psi_{L'}(i)$ in~(4.6),
$\Lambda_{\boldsymbol{f},L,k-1}(i)$ in~(4.9),
${R}^{+}_{\boldsymbol{f},L,k,\boldsymbol{g},K}(i)$ and ${R}^{++}_{\boldsymbol{f},L,k,K}(i)$
in~(5.8) and~(5.9), 
we have the parameters $|\Upsilon_{\boldsymbol{f}}(i)|$, $|\Psi_{L'}(i)|$ and
$\mathbbm{1}_{\{k\geqslant 1\}}\cdot|\Lambda_{\boldsymbol{f},L,k-1}(i)|$ in Claim~4.1.3
 satisfy
\begin{align*}
|\Upsilon_{\boldsymbol{f}}(i)|&=\sum_{(L, K)\in \mathcal{L}^2}|{R}^{+}_{\boldsymbol{f},L,e_L-2,\boldsymbol{f},K}(i)|,\\
|\Psi_{L'}(i)|&=\sum_{\substack{\boldsymbol{f},\boldsymbol{g}\in L'\cap {Q}(i)\\\ \boldsymbol{f}\neq \boldsymbol{g}}}
\sum_{(L,K)\in \mathcal{L}^2}|{R}^{+}_{\boldsymbol{f},L,e_L-2,\boldsymbol{g},K}(i)|,\\
\mathbbm{1}_{\{k\geqslant 1\}}\cdot|\Lambda_{\boldsymbol{f},L,k-1}(i)|&=
\mathbbm{1}_{\{k\geqslant 1\}}\cdot \sum_{K\in \mathcal{L}}|{R}^{++}_{\boldsymbol{f},L,k-1,K}(i)|.
\end{align*} 
By Remark~5.6 and Lemma~5.7, we have
\begin{align*}
|{R}^{+}_{\boldsymbol{f},L,e_L-2,\boldsymbol{f},K}(i)|&=O\bigl( n^{ \alpha' +r-3}(nt_{{M}})^k\bigr),\\
|{R}^{+}_{\boldsymbol{f},L,e_L-2,\boldsymbol{g},K}(i)|&=O\bigl( n^{ \alpha' +r-3}(nt_{{M}})^k\bigr),\\
\mathbbm{1}_{\{k\geqslant 1\}}\cdot|{R}^{++}_{\boldsymbol{f},L,k-1,K}(i)|&=O\bigl( n^{ \alpha' +(r-1)(e_L-k)-2}(nt_{{M}})^{k-1}\bigr).
\end{align*}
Note that $e_L, e_K\leqslant \ell$ 
 and $|\mathcal{L}|=O(1)$,
 we complete the proof of the equations (4.11)-(4.13) in Claim~4.1.3.
\end{proof}

\begin{proof}[Proof of (4.23) in Claim~4.1.6.]\
For any $\boldsymbol{f},\boldsymbol{g}\in {Q}(i)$, $L\in \mathcal{L}$
and  $0\leqslant k\leqslant e_L-2$, by the definitions of
$\Phi_{\boldsymbol{f},L,k,\boldsymbol{g}}(i)$ in Section~4.1.4 and
${R}^{+}_{\boldsymbol{f},L,k,\boldsymbol{g},K}(i)$
in~(5.8), 
we have $|\Phi_{\boldsymbol{f},L,k,\boldsymbol{g}}(i)|$ satisfies
\begin{align*}
|\Phi_{\boldsymbol{f},L,k,\boldsymbol{g}}(i)|&=\sum_{K\in \mathcal{L}}|{R}^{+}_{\boldsymbol{f},L,k,\boldsymbol{g},K}(i)|.
\end{align*} 
By Remark~5.6 and Lemma~5.7, 
\begin{align*}
|{R}^{+}_{\boldsymbol{f},L,k,\boldsymbol{g},K}(i)|=O\bigl( n^{ \alpha' +(r-1)(e_L-k)-r-1}(nt_{{M}})^k\bigr).
\end{align*}
Note that $e_K\leqslant \ell$ 
 and $|\mathcal{L}|=O(1)$, we complete the proof of the equation~(4.23) in Claim~4.1.6.
\end{proof}

\subsection{Lemma~5.4 and Lemma~5.7}

We will apply Lemma~5.1 to prove 
Lemma~5.4 and Lemma~5.7 below.

\begin{proof}[Proof of Lemma~5.4]\, For any $L\in\mathcal{L}$, 
let $H\subset L$ on vertex-set $V(H)=V(L)$ and $|H|=k$ with $0\leqslant k\leqslant e_L-2$.
Choose ${U}'\subseteq L\backslash H$ with $|{U}'|=|{U}|$ and a $m$-set
$\boldsymbol{f}'_m$ with
$\boldsymbol{f}'_m\subseteq\boldsymbol{f}'$ for some $\boldsymbol{f}'\in L\backslash (H\cup {U}')$.
Note that $\boldsymbol{f}'=\emptyset$ when  $m=0$. 

Define  $\hat{H}$ on the vertex-set $V(\hat{H})=V({U}')\cup \boldsymbol{f}'_m$
and $|\hat{H}|=0$. Let $\theta$ be an injection
 that maps the $r$-sets in ${U}'$ 
onto the $r$-sets in ${U}$, and $\boldsymbol{f}'_m$ onto  $\boldsymbol{f}_m$.
Since the number of ways to choose $H$, ${U}'$, $\boldsymbol{f}'_m$ 
and $\theta$ is finite, 
 we have $|\Gamma_{{U},\boldsymbol{f}_m,L,k}(i)|=O(|\Psi_{\theta,\hat{H},H}(i)|)$.
We are interested in applying Lemma~5.1 to prove an upper bound of $|\Psi_{\theta,\hat{H},H}(i)|$.

Consider any $\widetilde{H}$ with $\hat{H}\subseteq \widetilde{H}\subseteq H$. 
Note that  for $t\in[0,t_{{M}}]$ with $t_{{M}}$ in~(3.26), there exists $\alpha'\in (\alpha_0,\alpha)$
with $\alpha_0= \frac{\ell-2}{\ell-1}$ such that
\begin{align*}
n^\epsilon(nt)^{e_H-e_{\widetilde{H}}}\leqslant n^\epsilon(nt_M)^{\ell-2}=o(n^{\alpha'}).
\end{align*} 
By Lemma~5.1, it suffices to show
\begin{align}
&[v_H-(r-1)e_H]-[v_{\widetilde{H}}-(r-1)e_{\widetilde{H}}]\notag\\
&\leqslant (r-1)\cdot(e_L-k-|{U}|-\mathbbm{1}_{\{|{U}|+\mathbbm{1}_{\{m\geqslant 1\}}< e_L-k\}})-\mathbbm{1}_{\{m\geqslant 1\}}\cdot m\notag\\
&=\begin{cases}(r-1)\cdot(e_L-k-|{U}|-\mathbbm{1}_{\{|{U}|< e_L-k\}}),\ \text{if }m=0,\\
(r-1)\cdot(e_L-k-|{U}|-\mathbbm{1}_{\{|{U}|< e_L-k-1\}})-m,\ \text{otherwise},
\end{cases}
\end{align}
to complete our proof.

If $V(\widetilde{H})=V(H)$, then the trivial estimate 
\begin{align*}
[v_H-(r-1)e_H]-[v_{\widetilde{H}}-(r-1)e_{\widetilde{H}}]=(r-1)(e_{\widetilde{H}}-e_H)\leqslant 0
\end{align*}
implies the above inequality in~(5.10). 

Now suppose that
 $V(\widetilde{H})\subsetneq V(H)$. 
 Since  $H\subset L$ with
 $|H|=k$ for $0\leqslant k\leqslant e_L-2$, 
we have the number of edges of $L$ induced on $V(\widetilde{H})\cup \boldsymbol{f}'$
is at most $ \frac{v_{\widetilde{H}}+\mathbbm{1}_{\{m\geqslant 1\}}(r-m-1)}{r-1}$. 
In order to obtain the value of $e_{\widetilde{H}}$, 
note that ${U}'\cup \{\boldsymbol{f}'\}\subseteq L\backslash H$, then we delete
$|{U}|+\mathbbm{1}_{\{m\geqslant 1\}}$ from $ \frac{v_{\widetilde{H}}+\mathbbm{1}_{\{m\geqslant 1\}}(r-m-1)}{r-1}$.
If $|{U}|+\mathbbm{1}_{\{m\geqslant 1\}}<e_L-k$, 
then $\widetilde{H}$ consists of one linear path or several linear paths because $L$ is 
a  linear cycle of length at most $\ell$ by Remark~3.2. 
Then we further delete $\mathbbm{1}_{\{|{U}|+\mathbbm{1}_{\{m\geqslant 1\}}<e_L-k\}}$ 
from $e_{\widetilde{H}}$. 
Finally, we have the value of $e_{\widetilde{H}}$ satisfies
\begin{align*}
e_{\widetilde{H}}\leqslant \frac{v_{\widetilde{H}}+\mathbbm{1}_{\{m\geqslant 1\}}(r-m-1)}{r-1}-|{U}|-\mathbbm{1}_{\{m\geqslant 1\}}-\mathbbm{1}_{\{|{U}|+\mathbbm{1}_{\{m\geqslant 1\}}<e_L-k\}},
\end{align*}
which implies that
\begin{align*}
v_{\widetilde{H}}-(r-1)e_{\widetilde{H}}\geqslant (r-1)\cdot
(|{U}|+\mathbbm{1}_{\{|{U}|+\mathbbm{1}_{\{m\geqslant 1\}}<e_L-k\}})+\mathbbm{1}_{\{m\geqslant 1\}}\cdot m.
\end{align*}
Since  $v_L=(r-1)e_L$ because $L$ is 
a  linear cycle by Remark~3.2, we have $v_H=(r-1)e_L$.
Now $[v_H-(r-1)e_H]-[v_{\widetilde{H}}-(r-1)e_{\widetilde{H}}]$  satisfies the
inequality shown in~(5.10) by replacing  $e_H$ by $k$, and $v_{\widetilde{H}}-(r-1)e_{\widetilde{H}}$
by the above inequality.
\end{proof}

\vskip 0.3cm

\begin{proof}[Proof of Lemma~5.7]\, For any given $(L,K)\in \mathcal{L}^2$,
let $\bar{L}\subset L$ on vertex set $V(\bar{L})=V(L)$ and $|\bar{L}|=k$ with $0\leqslant k\leqslant e_L-2$,
and $\bar{K}\subset K$ with $V(\bar{K})=V(K)$ and $|\bar{K}|=e_K-2$,
such that there are $r$-sets $\boldsymbol{f}'$, $\boldsymbol{g}'$ and
$\boldsymbol{h}'$ satisfying $\bar{L}\cup \{\boldsymbol{f}',\boldsymbol{h}'\}\subseteq L$
and $\bar{K}\cup\{\boldsymbol{g}',\boldsymbol{h}'\}=K$,
where we allow $\boldsymbol{f}'=\boldsymbol{g}'$, while $\boldsymbol{f}'\neq \boldsymbol{h}'$ and 
$\boldsymbol{g}'\neq \boldsymbol{h}'$.

Consider $H=\bar{L}\cup \bar{K}$ with $V(H)=V(L)\cup V(K)$. 
Define $\hat{H}$  on vertex-set $V(\hat{H})=\boldsymbol{f}'\cup \boldsymbol{g}'\cup \boldsymbol{h}'$
and $|\hat{H}|=0$. Let $\theta$ be one injection with the property that 
$\boldsymbol{f}'$, $\boldsymbol{g}'$ and $\boldsymbol{h}'$ are mapped 
onto $\boldsymbol{f}$, $\boldsymbol{g}$ and $\boldsymbol{h}$.
Since  the number of ways to choose  $\bar{L}$, $\bar{K}$,  $\boldsymbol{f}'$, $\boldsymbol{g}'$, 
$\boldsymbol{h}'$ and $\theta$ is finite, we have
$|{R}_{\{\boldsymbol{f},\boldsymbol{h}\},L,k,\{\boldsymbol{g},\boldsymbol{h}\},K}(i)|
=O(|\Psi_{\theta,\hat{H},H}(i)|)$.
We are interested in applying Lemma~5.1 to show an upper bound 
of  $|\Psi_{\theta,\hat{H},H}(i)|$.

For any $\widetilde{H}$ with $\hat{H}\subseteq \widetilde{H}\subseteq H$, in Lemma~5.1,
note that for any
 $t\in[0,t_{{M}}]$
with $t_{{M}}$ in~(3.26), there exists $\alpha'\in (\alpha_0,\alpha)$ with $\alpha_0= \frac{\ell-2}{\ell-1}$
such that
\begin{align*}
n^\epsilon(nt)^{e_H-e_{\widetilde{H}}}\leqslant 
n^\epsilon(nt_M)^{k+e_K-2}=o\bigl(n^{\alpha'}(nt_{{M}})^{k}\bigr).
\end{align*}  
 It suffices to show
\begin{align}
[v_H-(r-1)e_H]-[v_{\widetilde{H}}-(r-1)e_{\widetilde{H}}]\leqslant (r-1)(e_L-k)-r-1
\end{align}
to complete our proof.

Let 
\begin{align}
A=\widetilde{H}\cap \bar{L}\quad\quad\text{and}\quad\quad B=\bar{K}\cap (\widetilde{H}\cup  \bar{L})
\end{align} with $V(A)=V(\widetilde{H})\cap V(\bar{L})$ and $V(B)=V(\bar{K})\cap (V(\widetilde{H})\cup  V(\bar{L}))$.
In fact, based on $H=\bar{L}\cup \bar{K}$ and $\widetilde{H}\subseteq H$, 
it follows that
\begin{align}
B=(\widetilde{H}\backslash \bar{L})\cup(\bar{K}\cap \bar{L}).
\end{align}
From the definitions of the set $A$ and $B$ in~(5.12) and~(5.13),
we have 
\begin{align*}
e_A+e_B&=e_{\bar{L}\cap \bar{K}}+e_{\widetilde{H}},\\
v_A+v_B&=v_{\bar{L}\cap \bar{K}}+v_{\widetilde{H}},
\end{align*}
 and then 
\begin{align}
&(v_H-v_{\widetilde{H}})-(r-1)(e_H-e_{\widetilde{H}})\notag\\
&=\bigl(v_{\bar{L}}+v_{\bar{K}}\bigr)-\bigl(v_A+v_B\bigr)-(r-1)\bigl(e_{\bar{L}}+e_{\bar{K}}\bigr)+(r-1)\bigl(e_A+e_B\bigr).
\end{align}

Note that $V(A)\subseteq V(L)$, we have $e_A\leqslant \frac{v_A-\mathbbm{1}_{\{v_A\neq v_L\}}}{r-1}$
by Remark~3.2.
Since $\boldsymbol{f}',\boldsymbol{h}'\subseteq V(A)$ while $\boldsymbol{f}'$,
$\boldsymbol{h}'\notin \bar{L}$, 
 we further have $e_A\leqslant \frac{v_A-\mathbbm{1}_{\{v_A\neq v_L\}}}{r-1}-2$,
which implies 
\begin{align}
v_A-(r-1)e_A\geqslant 2r-2+\mathbbm{1}_{\{v_A\neq v_L\}}.
\end{align} 
Similarly, since $V(B)\subseteq V(K)$, $\boldsymbol{g}'$, $\boldsymbol{h}'\subseteq V(B)$, 
$\boldsymbol{g}'$, $\boldsymbol{h}'\notin \bar{K}$,
we also have 
\begin{align}
v_B-(r-1)e_B\geqslant 2r-2+\mathbbm{1}_{\{v_B\neq v_{K}\}}.
\end{align}

Combined the equations shown in~(5.14)-(5.16) with
\begin{align*}
v_{\bar{L}}-(r-1)e_{\bar{L}}=(r-1)(e_L-k) \text{ and }v_{\bar{K}}=(r-1)(e_{\bar{K}}+2)
\end{align*}
because $L$, $K\in \mathcal{L}$, we have  the equation shown in~(5.11) holds
when $r\geqslant 3$ to complete the proof of Lemma~5.7.  
\end{proof}




\begin{thebibliography}{s2}

\bibitem{bal18}
 J. Balogh, B. Narayanan and J. Skokan, The number of hypergraphs
  without linear cycls. {\it J. Comb. Theory, Ser. B},
  {\bf 134} (2019), 309-321.


\bibitem{balgoh17}
J.~Balogh and L.~Li, On the number of linear hypergraphs of large girth.
\textit{J. Graph Theory}, {\bf 93}(1) (2020), 113-141.

\bibitem{behrend46}
F. A. Behrend, On sets of integers which contain no three terms in arithmetical 
progression. {\it Proceedings of the National Academy of Sciences and the United States
of America}, {\bf 32} (1946), 331.

\bibitem{bennett15}
P. Bennett and T. Bohman, A natural barrier in random greedy hypergraph matching.
{\it Combin. Probab. Computing}, {\bf 28} (2019), 816-825.

\bibitem{bohman09}
T. Bohman, The triangle-free process, {\it Adv. Math.}, {\bf 221} (2009), 1653-1677.

\bibitem{bohman101}
T. Bohman, A. Frieze and E. Lubetzky, A note on the random greedy triangle-packing algorithm.
{\it J. Comb.}, {\bf 1} (2010), 477-488.

  \bibitem{bohman15}
T. Bohman, A. Frieze and E. Lubetzky, Random triangle removal.
{\it Adv. Math.}, {\bf 280} (2015), 379-438.



 \bibitem{boh19}
  T. Bohman and L. Warnke. Large girth approximate Steiner triple systems.
  {\it J. Lond. Math. Soc.}, {\bf 100} (2019), 895-913.
  
  \bibitem{brown73}
  W. G. Brown, P. Erd\H{o}s and V. S\'{o}s, On the existence of trigulated spheres 
  in $3$-graphs and related problems. {\it Period. Math. Hung.}, {\bf 3} (1973), 221-228.

\bibitem{coll14}
  C. Collier-Cartaino, N. Graber and T. Jiang, Linear Tur\'{a}n numbers of 
   linear cycles and cycle-complete ramsey numbers. {\it Comb. Probab. Comput.},
   {\bf 27}(3) (2018), 358-386.
   
\bibitem{delcourt2022}
M. Delcourt and L. Postle, Finding an almost perfect matching in a hypergraph avoiding forbidden
submatchings.
arXiv:2204.08981.   
   
\bibitem{delcourt2024}
M. Delcourt and L. Postle, Proof of the high girth existence conjecture via refined absorption.
arXiv:2402.17856.


 \bibitem{erd86}
  P. Erd\H{o}s, P. Frankl and V. R\"{o}dl, The asymptotic number of graphs
  not containing a fixed subgraph and a problem for hypergraphs having no
  enponent. {\it Graphs Combinator.}, {\bf 2} (1986), 113-121.

\bibitem{ergem2019}
B. Ergemlidze, E. Gy\H{o}ri and A. Methuku, Asymptotics for Tur\'{a}n number
of cycles in $3$-uniform linear hypergraphs. {\it J.
Comb. Theory, Ser. A}, {\bf 163} (2019), 163-181.

 \bibitem{freed75}
 D. A. Freedman, On tail probabilities for martingales. {\it Ann. Probability}, {\bf 3} (1975),
 100-118.
 
 \bibitem{fure13}
  Z. F\"{u}redi and M. Ruszink\'{o}. Uniform hypergraphs containing no grids.
  {\it Adv. Math.}, {\bf 240} (2013), 302-324.
  
  
  \bibitem{gao2021}
  G. Gao and A. Chang. A linear hypergraph extension of the bipartite Tur\'{a}n
  problem. {\it Eur. J. Comb.}, {\bf 93} (2021), Paper No. 103269.

  \bibitem{gao2023}
G. Gao, A. Chang and Q. Sun, Asymptotic Tur\'{a}n number for linear $5$-cycle in $3$-uniform
linear hypergraphs. {\it Discrete Math.}, {\bf 346} (2023), Paper No. 113128.
  
  \bibitem{glock20}
  S. Glock, D. K\"{u}hn, A. Lo and D. Osthus. On a conjecture of 
  Erd\H{o}s on locally sparse Steiner triple systems. {\it Combinatorica},
  {\bf 40} (2020), 363-403.
  
\bibitem{glock2024}
  S. Glock, F. Joos, J. Kim, M. K\"{u}hn and L. Lichev. Conflict-free hypergraph
  matchings. {\it J. Lond. Math. Soc.},
  {\bf 109}(2) (2024), Paper No. e12899.
  
  \bibitem{grable97}
D. Grable, On random greedy triangle packing.  {\it Electron. J. Comb.}, {\bf 4} (1997), \#R11.
  

  
  \bibitem{jiang30}
T. Jiang, J. Ma and L. Yepremyan, Linear cycles of consecutive lengths. {\it J.
Comb. Theory, Ser. B}, {\bf 163} (2023), 1-24. 

\bibitem{joos2025}
F. Joos and M. K\"{u}hn, The hypergraph removal process. arXiv:2412.15039.
 
 
   \bibitem{kee20}
  P. Keevash and J. Long. The Brown-Erd\H{o}s-S\'{o}s conjecture for hypergraphs of 
  large uniformity. Proc. Amer. Math. Soc. (to appear).
  

  
  \bibitem{kuhn16}
D. K\"{u}hn, D. Osthus and A. Taylor, On the random greedy $F$-free hypergraph process, {\it SIAM J. Discrete Math.},
{\bf 30}(3) (2016), 1343-1350.

\bibitem{kwan2022}
M. Kwan, A. Sah, M. Sawhney and M. Simkin,
High-girth Steiner triple systems. arXiv: 2201.04554.
  
  \bibitem{laz03}
  F. Lazebnik and J. Verstra\"{e}te, On hypergraphs of girth five.
  {\it Electron. J. Comb.}, {\bf 10} (2003), 25.


  
\bibitem{rw92}
A. Ruci\'{n}ski and N. C. Wormald. Random graph processes with degree restrictions. {\it Combin.
Probab. Comput.}, {\bf 1}(2) (1992), 169-180.

\bibitem{rw97}
A. Ruci\'{n}ski and N. C. Wormald. Random graph processes with maximum degree $2$. {\it Ann.
Appl. Probab.}, {\bf 7}(1) (1997), 183-199.
  
  \bibitem{ruz78}
  I. Z. Ruzsa and E. Szemer\'{e}di, Triple systems with no six points carrying three 
  triangles. {\it Combinatorics (Keszthely, 1976),
  Coll. Math. Soc. J. Bolyai}, {\bf 18}(1978), 939-945.
  
\bibitem{tian23}
F. Tian, Z. L. Liu and X. F. Pan, On the random greedy linear uniform hypergraph packing.
{\it Australas. J. Comb.}, {\bf 87}(3) (2023), 365-390.
  
 

\bibitem{warnke16}
 L. Warnke, On the method of typical bounded differences. {\it Combin. Probab. Comput.},
 {\bf 25} (2016), 269-299. 
\end{thebibliography}
\end{document}